\documentclass[12pt,reqno,final
]{amsart}
\usepackage[left=1.12in,right=1.25in]{geometry}
\usepackage{amsfonts,amssymb,amsmath}
\usepackage%[demo]
{graphicx}
\usepackage[comma, authoryear]{natbib}
\usepackage[colorlinks]{hyperref}

\usepackage{xspace,colortbl}
\usepackage{color}

\usepackage[color,notref,notcite]{showkeys}
\usepackage{enumerate}
\definecolor{labelkey}{rgb}{0.6,0,1}

\newcommand{\inw}[1]{\stackrel{\circ}{#1}}

\def\calf{\mathcal{F}}

\def\R{\mathbb{R}}

\def\A{\mathcal{A}}

\def\L{\mathcal{L}}

\def\I{\mathbb{I}}

\def\cH{\mathcal{H}}

\numberwithin{equation}{section}
\newtheorem{thm}{Theorem}[section]
\newtheorem{lem}[thm]{Lemma}

\newtheorem{prop}[thm]{Proposition}
\newtheorem{RM}[thm]{Remark}

\newtheorem{algo}{Algorithm}

\makeatother

\begin{document}

\title[Option pricing and hedging for RSGBM models]{Option pricing and hedging for regime-switching geometric Brownian
motion models}

\author{Bruno R\'{e}millard}

\address{CRM, GERAD, and Department of Decision Sciences, HEC Montr\'{e}al, 3000 chemin de la C\^{o}te Sainte-Catherine,
Montr\'{e}al (Qu\'{e}bec) Canada H3T 2A7}

\email{bruno.remillard@hec.ca}

\author{Sylvain Rubenthaler}

\address{Laboratoire J. A. Dieudonn\'{e}, Universit\'{e} Nice Sophia Antipolis,
Parc Valrose, 06108 Nice cedex 2, France}

\email{rubentha@unice.fr}

\thanks{Partial funding in support of this work was provided by the Natural
Sciences and Engineering Research Council of Canada. The authors thank Hugo Lamarre for his helpful suggestions.}
\begin{abstract}
We find the variance-optimal equivalent martingale measure
when multivariate assets are modeled by a  regime-switching geometric Brownian motion, and the regimes are represented by a homogeneous continuous time Markov chain. Under this new measure, the Markov chain driving the regimes is no longer homogeneous, which differs  from the equivalent martingale measures usually proposed in the literature.
 We show the solution minimizes the mean-variance hedging error  under the objective measure.  As argued by \citet{Schweizer:1996}, the variance-optimal equivalent measure naturally extends canonical option pricing results to the case of an incomplete market and the expectation under the proposed measure may be interpreted as an option price. Solutions for the option
value and the optimal hedging strategy are easily obtained from Monte Carlo simulations. Two applications are considered.
\end{abstract}
%We also find the solution minimizing
% As shown by \citet{Schweizer:1996}, the associated option price is given by the expectation under the variance-optimal  equivalent martingale measure. The optimal solution is easy to implement since Monte Carlo simulations can be used to obtain the option value and the optimal hedging strategy. Two examples of application are considered.


\maketitle
%\subjclass{Primary 91G20, Secondary 91G80, 91G60.}

%\tableofcontents{}

%%%%%%%%%%%%%%%%%%%%%%%%%%%%%%%%%%%%%%%%%%%%%%%%%%%%%%%%%%%%%%%%%%%%%%%%%%%%%%%%
%%%%%%%%%%%%%%%%%%%%%%%%%%%%%%%%%%%%%%%%%%%%%%%%%%%%%%%%%%%%%%%%%%%%%%%%%%%%%%%%
\section{Introduction}
%%%%%%%%%%%%%%%%%%%%%%%%%%%%%%%%%%%%%%%%%%%%%%%%%%%%%%%%%%%%%%%%%%%%%%%%%%%%%%%%
%%%%%%%%%%%%%%%%%%%%%%%%%%%%%%%%%%%%%%%%%%%%%%%%%%%%%%%%%%%%%%%%%%%%%%%%%%%%%%%%
\setcounter{equation}{0}

Recently, regime-switching geometric Brownian motion
(RSGBM) models have been the subject of much attention. They offer more flexibility than geometric Brownian motion models in capturing reported asset dynamics features such as time-varying volatility and high-order moments.

When the market is  incomplete, as  is the case for RSGBM models,
there are infinitely many  equivalent martingale measures to choose from, all free of arbitrage opportunities. Hence, arbitrage considerations alone cannot identify the measure and one
has then to choose the ``best'' martingale measure according to some  other
criterion.

In many articles, e.g., \citet{Guo:2001} and
\citet{Shen/Fan/Siu:2014} to name a few,  options are evaluated under a parametrized family of martingale measures which are typically calibrated to historical option prices for empirical applications.
 However, the class of martingale measures proposed so far   in the literature is quite limited. For example, in most
cases, the distribution of the hidden
regimes under the equivalent martingale  measure is either unchanged
or  remains a time homogeneous Markov chain.

%options are evaluated under a given family of martingale measures depending on parameters than are obtained via calibration.

 The main objective of this paper is to find  (1) an equivalent
martingale measure that is optimal in the sense of \citet{Schweizer:1996}, and  (2)  the   optimal hedging strategy.
More precisely, we find the equivalent signed martingale measure minimizing the variance of its density with respect to the objective measure, i.e.\ the so-called variance-optimal equivalent martingale measure. This first optimization problem is defined precisely in Section \ref{ssec:opt-var}, while the second optimization problem pertaining to the mean-variance hedging is defined in Section \ref{ssec:hedging}.
When prices are continuous, the variance-optimal measure is in fact a probability measure  and expectations can be interpreted as arbitrage-free prices arising from a hedging strategy minimizing a mean-variance criterion for the (discounted) hedging error.

% as defined in \citet{Schweizer:1996}
%In full generality, this problem is to find, among all equivalent martingale signed measures, the one minimizing the variance of its density with respect to the objective measure. When prices are continuous, which is the case here, this signed measure is in fact a positive measure \citep{Schweizer:1996}. As for the hedging error,
%the proposed measure of
%quality is the mean-variance of the (discounted)
%hedging error. As a consequence of \citet{Schweizer:1996}, the resulting  option valuation is in
%fact the one given by the variance-optimal measure.

Our work is related to \citet{Cerny/Kallsen:2007} who show the existence of a solution for the optimal hedging problem under a very large class of models.  For most cases of interest, their proposed solution is very difficult to find  in practice. So far only a few cases have been solved in continuous time,
including a very special stochastic volatility model
\citep{Cerny/Kallsen:2008a}.  For an implementable solution in
discrete time, see, e.g., \citet{Remillard/Rubenthaler:2013}

 %Basically, we take the same class as in
%\citet{Cerny/Kallsen:2007} where the optimal hedging problem was
%shown to have a solution in a very general setting. However, in most
%cases of interest, their proposed solution is very difficult to
%find. So far only a few cases have been solved in continuous time,
%including a very special stochastic volatility model
%\citep{Cerny/Kallsen:2008a}. For an implementable solution in
%discrete time, see, e.g., \citet{Remillard/Rubenthaler:2013}.

%When prices are continuous, which is the case here, this signed measure is in fact a positive measure \citep{Schweizer:1996}

In the next section, we describe the model and state the hedging  problem. As in
\citet{Shen/Fan/Siu:2014}, the instantaneous interest rate may vary with  regimes.
Then in Section \ref{sec:emm}, we propose a change of
measure $\tilde P$, equivalent to the objective measure $P$, for which the discounted prices of  assets are  martingales. Under this martingale measure, the distribution of
regimes becomes a time inhomogeneous Markov chain. We also define the forward measure and give
representations for the associated option prices. The main results is that $\tilde P$ is  variance-optimal. Next, in Section
\ref{sec:hedging}, we define   a hedging strategy for the option
price associated with the change of measure and we show  that this
hedging strategy minimizes the mean-variance hedging error. Finally, in Section \ref{sec:application}, examples of
computations are given, based on simulations and Fourier transforms.

%%%%%%%%%%%%%%%%%%%%%%%%%%%%%%%%%%%%%%%%%%%%%%%%%%%%%%%%%%%%%%%%%%%%%%%%%%%%%%%%
%%%%%%%%%%%%%%%%%%%%%%%%%%%%%%%%%%%%%%%%%%%%%%%%%%%%%%%%%%%%%%%%%%%%%%%%%%%%%%%%
\section{Regime-switching model}\label{sec:model}
%%%%%%%%%%%%%%%%%%%%%%%%%%%%%%%%%%%%%%%%%%%%%%%%%%%%%%%%%%%%%%%%%%%%%%%%%%%%%%%%
%%%%%%%%%%%%%%%%%%%%%%%%%%%%%%%%%%%%%%%%%%%%%%%%%%%%%%%%%%%%%%%%%%%%%%%%%%%%%%%%

We first define the model and
then we state the hedging problem.

\subsection{Model}\label{ssec:model}

Let $\tau$ be a continuous time Markov chain on
$\{1,\ldots,l\}$, with infi\-ni\-tesimal generator $\Lambda$. In
particular, $P(\tau_{t}=j|\tau_{0}=i)=P_{ij}(t)$, where the
transition matrix $P$ can be written as $P(t)=e^{t\Lambda}$,
$t\ge0$. Then, the (continuous) price process $S$ is defined as the
solution of the stochastic differential equation
\begin{equation}
dS_{t}=D(S_{t})\mu(\tau_{t})dt+D(S_{t})\sigma(\tau_{t-})dW_{t},\label{eq:sde}
\end{equation}
where $D(s)$ is the diagonal matrix with diagonal elements
$(s_{j})_{j=1}^{d}$ and $W$ is a $d$-dimensional Brownian motion,
independent of $\tau$. We set $(\mathcal{F}_{t}^{\tau})_{t\geq0}$
for the filtration generated by $\tau$ and
$(\mathcal{F}_{t}^{W})_{t\geq0}$ for the filtration generated by
$W$. It is assumed that $a(i)=\sigma(i)\sigma(i)^{\top}$ is
invertible for any $i\in\{1,\ldots,l\}$.
%It follows from
%\eqref{eq:sde} that for any $j\in\{1,\ldots,d\}$,
%\begin{equation}\label{eq:sdemargins}
%dS_{t}^{(j)}=S_{t}^{(j)}\mu^{(j)}(\tau_{t})dt+S_{t}^{(j)}\sum_{k=1}^{d}\sigma_{jk}(\tau_{t-})dW_{t}^{(k)}.
%\end{equation}
If $T_{\alpha}$ represents the time of the $\alpha-th$ jump of the
regime process $\tau$, with $T_{0}=0$, then $S$ has the following
representation for any $t\in[T_{\alpha-1},T_{\alpha}]$, $\alpha\ge1$,
given $\tau_{T_{\alpha-1}}=i$:
\begin{equation}
S_{t}^{(j)}=S_{T_{\alpha-1}}^{(j)}e^{\upsilon^{(j)}(i)(t-T_{\alpha-1})+\sum_{k=1}^{d}\sigma_{jk}(i)\left(W_{t}^{(k)}-W_{T_{\alpha-1}}^{(k)}\right)},\quad j\in\{1,\ldots,d\},\label{eq:sdesol}
\end{equation}
where $\upsilon(i)=\mu(i)-\frac{1}{2}{\rm diag}\{a(i)\}$, and ${\rm
diag}(b)$ is the vector formed with the diagonal elements of the
matrix $b$. Because of \eqref{eq:sdesol}, such a process is called a
($d$-dimensional) regime-switching geometric Brownian motion with
parameters $\mu(i),a(i)$, $i \in \{1,\ldots,l\}$. Note that $S$ is
not a Markov process in general; however $(S,\tau)$ is a Markov
process.
%\begin{RM}\label{rem:INHOMOGENEOUS}
One could also consider inhomogeneous Markov chains $\tau$, with an
infinitesimal generator depending on time. See Appendix
\ref{app:sim} for a possible construction.
%\end{RM}
One may  also define the law of $(S,\tau)$ through its
infinitesimal generator $\mathcal{H}$, defined  for all $f$ is in $C_{b}^{2}$ (the space of bounded continuous
functions with bounded and continuous derivatives of order one and
two) by
\begin{equation}\label{eq:L}
\mathcal{H}f(s,i)=\L_{i}f(s,i)+\sum_{j=1}^{l}\Lambda_{ij}f(s,j),
\end{equation}
where for each $i=1,\ldots,l$ , $\L_{i}$ is the infinitesimal generator
associated with the geometric Brownian motion with parameters $\mu(i),\sigma(i)$,
defined by
\[
\L_{i}f(s)=\mu(i)^{\top}D(s)\nabla f(s)+\frac{1}{2}\sum_{k=1}^{d}\sum_{j=1}^{d}a_{kj}(i)s_{k}s_{j}\partial_{s_{k}}\partial_{s_{j}}f(s).
\]
The main property of the infinitesimal generator $\L$ of a
Markov process $x_{t}$ that will be used is that
$f(x_{t})-\int_{0}^{t}\L f(x_{u})du$ is a martingale. For more
details, see, e.g., \citet{Ethier/Kurtz:1986}.
To be able to compare our results with those of \citet{Shen/Fan/Siu:2014}, we assume that the instantaneous interest
rate at time $t$ is given by $r_{\tau_{t-}}$, with $r\in [0,\infty)^{l}$. The  (random) discounting factor $B$ is
then given by
\begin{equation}\label{eq:discount}
B_{t}=e^{-\int_{0}^{t}r_{\tau_{u}}du}, \quad t\ge 0,
\end{equation}
so that $X_t = B_t S_t$ is the discounted price process.

We now describe the two optimization problems  we want to solve in this paper.

\subsection{Variance-optimal  martingale measure}\label{ssec:opt-var}
As defined in \citet{Schweizer:1996},
let $\mathcal{Q}$ be the set of all absolutely continuous signed martingale measures $Q$ with respect to $P$, having a square integrable density $Z_t = \left. \frac{dQ}{dP}\right|_{\calf_t}$, so that  for all $t\in [0,T]$, $E(Z_t)=1$, and $X_tZ_t$ is a $P$-martingale. Then $\tilde P \in \mathcal{Q}$ is variance-optimal if
$$
E\left[\left(\left.\frac{d\tilde P}{dP}\right|_{\calf_T}\right)^2 \right] = \inf_{Q\in \mathcal{Q}}E\left[\left(\left.\frac{d\tilde P}{dP}\right|_{\calf_T}\right)^2 \right].
$$
In \citet{Schweizer:1996}, it was shown that the variance-optimal measure exists, is unique and it is indeed a true probability measure if $X$ is continuous, which is the case here.

One of the main aim of this paper is to find the variance-optimal martingale measure for the RSGBM model. This is done in Lemma \ref{lem:changeofmeasure}.

\subsection{Hedging problem}\label{ssec:hedging}
The other motivation of the paper is to solve a quadratic hedging problem for the price process $S$. More precisely, let
 $\A$ be the class  of admissible strategies $\psi$ defined as
the $L^{2}(P)$ closure of (finite) linear combination of simple
admissible strategies $\psi$ of the form $\psi=Y\I_{(T_{1},T_{2}]}$,
where $T_{1},T_{2}$ are stopping times so that $0\le T_{1}\le
T_{2}\le T$, and $Y$ is a bounded random vector that is
$\calf_{T_{1}}$-measurable. Note that the class of admissible
strategies is a bit different from the one considered by
\citet{Cerny/Kallsen:2007}.  The (quadratic)
hedging problem for the regime-switching geometric Brownian motion is to minimize the discounted quadratic hedging error defined by
\begin{equation}\label{eq:problem}
\mathit{HE}(\pi,\psi) =  E\left[\left\{ B_{T}\Phi(S_{T})-\pi-\int_{0}^{T}\psi_{u}^{\top}dX_{u}\right\} ^{2}\right], \quad (\pi,\psi)\in\R\times\mathcal{A},
\end{equation}
where $X_t = B_t S_t$ is the discounted price process. Here, $\pi$ is the initial value of the portfolio, and $\pi+\int_{0}^{T}\psi_{u}^{\top}dX_{u}$ is its discounted value at time $t\ge 0$.

In the next section, we propose a candidate for the variance-optimal equivalent martingale measure. It was shown in \citet{Schweizer:1996} that for the variance-optimal equivalent martingale measure $\hat P$, if for a given $\Phi$, $(C_0,\varphi)$ minimizes \eqref{eq:problem}, then $C_0 = E^{\hat P}[B_T \Phi(S_T)]$.

\section{Change of measures and off-line computations}\label{sec:emm}

In order to be able to define the proposed change of measures, we need to introduce auxiliary deterministic functions. This is done next in Section \ref{ssec:offline}, after introducing first some notations.

For $i\in\{1,\ldots,l\}$, let $m(i)=(\mu(i)-r_{i}\mathbf{1})$,
$\rho(i)=a(i)^{-1}m(i)$, and set
\begin{equation}\label{eq:ell}
\ell_{i}=\rho(i)^{\top}m(i)=m(i)^{\top}a(i)^{-1}m(i) = \rho(i)^{\top}a(i)\rho(i) \ge 0.
\end{equation}
As a result, one obtains the following representation for the discounted
value $X$ of $S$:
\begin{equation}\label{eq:Xmart}
X_{t}=B_{t}S_{t}=S_{0}+\int_{0}^{t}D(X_{u})m(\tau_{u})du+\int_{0}^{t}D(X_{u})\sigma(\tau_{u-})dW_{u}, \quad t\ge 0.
\end{equation}

%Next, to define a regime-switching Brownian motion,  The process
%$B_t$ is a regime-switching Brownian motion  with parameters
%$(\upsilon(i),a(i))$, $i=1,\ldots, l$, and $\Lambda$, if the process
%$(B,\tau)$ is a Markov process with infinitesimal generator
%$$
%\L f(s,i) = \L_{\upsilon(i),a(i)} f(s,i)+ \sum_{j=1}^l
%\Lambda_{ij}f(s,j),
%$$

%We are now in a position to state some properties of the Markov
%process $(S,\tau)$. First, note that $S$ and $\tau$ are
%semimartingales. In fact, if $g(s,j) = s$, then $\L g(s,i) =
% D(s)\mu(i)$, so
%\begin{equation}\label{eq:Smart}
%M^{(g)}_t = S_t -S_0 -\int_0^t D(S_u)\mu(\tau_u)du= \int_0^t
%D(S_u)\sigma(\tau_{u-})dW_u
%\end{equation}
%is a martingale and $S$ is a semi-martingale. %As a result, one
%obtains the following representation for the discounted value $X$ of
%$S$:
%\begin{equation}\label{eq:Xmart}
%X_t = B_tS_t = S_0  +\int_0^t D(X_u)m(\tau_u)du+ \int_0^t
%D(X_u)\sigma(\tau_{u-})dW_u.
%\end{equation}
%Next, setting $h(s,j)=j$, one obtains that
%$$
%M^{(h)}_t = \tau_t -\tau_0 -\int_0^t \Lambda h(\tau_u)du
%$$
%is a martingale. Moreover, it follows from Lemma \ref{lem:crochet}
%that $ \left[M^{(g)},M^{(h)}\right]_t$ is a martingale.

We now define the deterministic functions that plays a central role in the change of maesures.
\subsection{Auxiliary functions}\label{ssec:offline}

Let $\gamma(t)=e^{t\{\Lambda-D(\ell)\}}\mathbf{1}$ and $\delta(t)=e^{t\{\Lambda-D(\ell+r)\}}\mathbf{1}$. Then, for
all $i\in\{1,\ldots,l\}$, $\gamma_{i}(0)=1$, $\delta_i(0)=1$, and
\begin{eqnarray}
\dot{\gamma}_{i}(t)&=&\frac{d}{dt}\gamma_{i}(t)=-\ell_{i}\gamma_{i}(t)+\sum_{j=1}^{l}\Lambda_{ij}\gamma_{j}(t),\label{eq:odegamma}\\
\dot{\delta}_{i}(t)&=&\frac{d}{dt}\delta_{i}(t)=-(\ell_{i}+r_i)\delta_{i}(t)+\sum_{j=1}^{l}\Lambda_{ij}\delta_{j}(t)\label{eq:odedelta}.
\end{eqnarray}
The following lemma, proven in  Appendix \ref{app:lemgamma}, gives other representations for $\gamma$ and $\delta$.

\begin{lem}\label{lem:gamma} Let $\gamma$  and $\delta$ be the solutions of \eqref{eq:odegamma} and \eqref{eq:odedelta}.
Then for all
$t\ge0$ and for all $i\in\{1,\ldots,l\}$,  $1\ge\gamma_{i}(t)\ge e^{(\Lambda_{ii}-\ell_{i})t}>0$, $1\ge\delta_{i}(t)\ge e^{(\Lambda_{ii}-\ell_{i}-r_i)t}>0$, and
\begin{eqnarray}
\gamma_{i}(t) &=&E_{i}\left[
e^{-\int_{0}^{t}\ell_{\tau_{u}}du}\right], \label{eq:odegammarep}\\
\delta_{i}(t) &=&E_{i}\left[
e^{-\int_{0}^{t}\left(r_{\tau_u}+\ell_{\tau_{u}}\right)du}\right]. \label{eq:odedeltarep}
\end{eqnarray}

\end{lem}

%\begin{RM}\label{rem:Z2} As a consequence of \eqref{eq:odegammarep},
%\begin{equation}\label{eq:Z1}
%Z^{(1)}(t)=E\left[\left. e^{-\int_{0}^{T}\ell_{\tau_{u}}du}\right|\calf_{t}\right] =\gamma_{\tau_{t}}(T-t)e^{-\int_{0}^{t}\ell_{\tau_{u}}du},\quad t\in[0,T],
%\end{equation}
%is a positive martingale with initial value $\gamma_{\tau_{0}}(T)$.
%\end{RM}
Next, by Lemma \ref{lem:gamma}, for any $t\ge 0$ and any $i\in \{1,\ldots,l\}$, $\gamma_i(t) >0$ and $\delta_i(t) >0$, so one may define $\beta_{i}(t) = \delta_{i}(t)/\gamma_{i}(t)$, and
\begin{eqnarray}
(\tilde{\Lambda}_{t})_{ij} &=&\Lambda_{ij}\gamma_{j}(t)/\gamma_{i}(t),\quad i\neq j,\qquad(\tilde{\Lambda}_{t})_{ii}=-\sum_{j\neq i}(\tilde{\Lambda}_{t})_{ij},\label{eq:lambdat}\\
\left(\inw\Lambda_{t}\right)_{ij}&=&\frac{\Lambda_{ij}\delta_{j}(t)}{\delta_{i}(t)}=\frac{(\tilde \Lambda_t)_{ij}\beta_{j}(t)}{\beta_{i}(t)},\quad i\neq j,\qquad \left(\inw\Lambda_{t}\right)_{ii}=-\sum_{j\neq
i}\left(\inw\Lambda_{t}\right)_{ij}. \label{eq:lambdat0}
\end{eqnarray}
Then $\tilde{\Lambda}_{t}$, $t\in[0,T]$, is the infinitesimal generator
of a time inhomogeneous Markov chain. The same is true for $\inw{\Lambda}_{t}$, $t\in[0,T]$.

\begin{RM}\label{rem:beta} It follows from Lemma \ref{lem:gamma} that $\beta_{i}(0)=1$, and that for any $t>0$,
\begin{equation}
\beta_{i}(t)=\frac{E_{i}\left[ e^{-\int_{0}^{t}(r_{\tau_{u}}+\ell_{\tau_{u}})du}\right] }{\gamma_{i}(t)},\label{eq:beta1}
\end{equation}
and
\begin{equation}
\dot \beta_i(t) = \frac{d}{dt}\beta_{i}(t)=-r_{i}\beta_{i}(t)+(\tilde{\Lambda}_{t})_{ii}\beta_{i}(t)-\left(\inw\Lambda_{t}\right)_{ii}\beta_{i}(t), \quad t>0.\label{eq:betaedp}
\end{equation}
Note that if $r$ is constant,
then $\beta_{i}(t)=B_{t}$ and $\inw\Lambda_{t}=\tilde{\Lambda}_{t}$ for any $t\ge 0$.
\end{RM}

%%%%%%%%%%%%%%%%%%%%%%%%%%%%%%%%%%%%%%%%%%%%%%%%%%%%%%%%%%%%

\subsection{Change of measure}

\label{ssec:change} %%%%%%%%%%%%%%%%%%%%%%%%%%%%%%%%%%%%%%%%%%%%%%%%%%%%%%%%%%%%

Set $M_{t}=\int_{0}^{t}\rho(\tau_{u-})^{\top}D^{-1}(X_{u})dX_{u}$, $t\ge 0$. Then,
\begin{equation}\label{eq:M}
M_{t}=\int_{0}^{t}\ell_{\tau_{u}}du+\int_{0}^{t}\rho(\tau_{u-})^{\top}\sigma(\tau_{u-})dW_{u}=\int_{0}^{t}\ell_{\tau_{u}}du+\mathcal{M}_{t},\quad t\ge 0,
\end{equation}
where $\mathcal{M}$ is a martingale with quadratic variation
$[\mathcal{M},\mathcal{M}]_t = [M,M]_t= \int_0^t \ell_{\tau_u}du$, using \eqref{eq:ell} and \eqref{eq:Xmart}.  For $t\ge 0$, further set
$$
Z_t=\mathcal{E}_t\left\{ -M\right\}  = e^{-M_t -
\frac{1}{2}[M,M]_t} = \mathcal{E}_t\left\{ -\mathcal{M}\right\} e^{-\int_0^t \ell_{\tau_u}du} = Z_t^{(1)} e^{-\int_0^t \ell_{\tau_u}du},
$$
where $\mathcal{E}$ is the stochastic
exponential; see, e.g., \citet[page 85]{Protter:2004}.
We are now in a position to define the change of measure. The proof is given in Appendix \ref{app:changeofmeasure0}.
Note that this change of measure can be obtained as a limit of the discrete case \citep{Schweizer:1995a,Remillard/Rubenthaler:2013} when the number of hedging periods tends to infinity.

\begin{lem}\label{lem:changeofmeasure0}
 $Z$ is a multiplicative functional, and the
processes $Z_t^{(1)} = e^{-\mathcal{M}_t -
\frac{1}{2}[\mathcal{M},\mathcal{M}]_t}$ and
$Z_{t}^{(2)}=\gamma_{\tau_{t}}(T-t)e^{-\int_{0}^{t}\ell_{\tau_{u}}du}$
are positive orthogonal martingales. Moreover,
$E(Z_{T}|\calf_{t})=Z_{t}^{(1)}Z_{t}^{(2)}=Z_{t}\gamma_{\tau_{t}}(T-t)$ and $\gamma_{i}(t)=E_{i}(Z_{t})$, for all $t\ge0$.  Also,
$\left.\frac{d\tilde{P}}{dP}\right|_{\calf_{T}}=Z_{T}/\gamma_{\tau_{0}}(T)=Z_{T}^{(1)}Z_{T}^{(2)}/\gamma_{\tau_{0}}(T)$ defines an equivalent change of measure, and
for any $t\in [0,T]$,
\begin{equation}\label{eq:Ptilde}
\left.\frac{d\tilde{P}}{dP}\right|_{\calf_{t}}=Z_{t}\gamma_{\tau_{t}}(T-t)/\gamma_{\tau_{0}}(T)=Z_{t}^{(1)}Z_{t}^{(2)}/\gamma_{\tau_{0}}(T).
\end{equation}
Under $\tilde P$,
$\tau$ is an inhomogeneous Markov chain with generator $\tilde \Lambda_{T-t}$, $\tilde W_t = W_t +\int_0^t \sigma_{\tau_u}^\top \rho_{\tau_u} du$ is a Brownian motion independent of $\tau$, and $X$ is a $\tilde P$-martingale. As a result,   under $\tilde P$,
$(S,\tau)$ has infinitesimal generator $\tilde{\cH}_{t}$  given by
\begin{equation}\label{eq:genHHMMnew}
{\tilde{\mathcal{H}}}_{t}f(s,i)=\tilde{\L}_{i}f(s,i)+\sum_{j=1}^{l}(\tilde{\Lambda}_{T-t})_{ij}f(s,j),\quad t\in [0,T],
\end{equation}
with
\begin{equation}\label{eq:riskneutral-BMnew}
\tilde{\L}_{i}f(s,i)=r_{i}\sum_{k=1}^{d}s_{k}\partial_{s_{k}}f(s,i)+\frac{1}{2}\sum_{j=1}^{d}\sum_{k=1}^{d}a_{jk}(i)s_{j}s_{k}\partial_{s_{j}}\partial_{s_{k}}f(s,i).
\end{equation}
\end{lem}
\begin{RM}\label{rem:schweizer-follmer} Note that $Z^{(1)}$ corresponds
to the density of the so-called ``minimal martingale measure'' \citet{Follmer/Schweizer:1991},
which is different from the (unnormalized) density $Z^{(1)}Z^{(2)}$.
This is a non-trivial example of the difference between local risk-minimizing
strategies, as considered in \citet{Follmer/Schweizer:1991}, and
global risk-minimizing strategies, as considered here.
\end{RM}
The main result of this section is stated next, and it is proven in Appendix
\ref{app:changeofmeasure}.
\begin{lem}\label{lem:changeofmeasure}
%Set $Y_{t}= S_{t} Z_{t} \gamma_{\tau_{t}}(T-t)$.
%Then  both  $Y_{t}-Y_{0}-\int_{0}^{t}r_{\tau_{u}}Y_{u}du$ and
%$B_{t}Y_{t}=X_{t} Z_{t}\gamma_{\tau_{t}}(T-t)$ are martingales.
%In particular, $X_t$, $t\in [0,T]$, is a $\tilde P$-martingale.
%Moreover,
$\tilde P$ is the  variance-optimal equivalent martingale measure, as defined in Section \ref{ssec:opt-var}.
\end{lem}

The following proposition defines the so-called forward  measure, which is necessary only when the interest rate is stochastic. The proof of the proposition is similar to the proof of Lemma \ref{lem:changeofmeasure0} so it is omitted.
%Its proof is given in Appendix \ref{pf:propbeta}.

\begin{prop}\label{prop:forward}
For any $t\in [0,T]$,
%\begin{equation}\label{eq:beta}
$ E^{\tilde{P}}\left[ \left.  B_T \right|\calf_{t}\right] = B_t \beta_{\tau_{t}}(T-t)$.
%\end{equation}
As a result,
\begin{equation}\label{eq:forward}
\left.\frac{d\inw
P}{d\tilde{P}}\right|_{\calf_{T}}= B_T/\beta_{\tau_{0}}(T)
\end{equation}
defines the forward measure, and for any $t\in[0,T]$,
$
\left.\frac{d\inw
P}{d\tilde{P}}\right|_{\calf_{t}}=\frac{\beta_{\tau_{t}}(T-t)}{\beta_{\tau_{0}}(T)}
B_t$.
Under $\inw P$,
$\tau$ is an inhomogeneous Markov chain with generator $\inw \Lambda_{T-t}$, $\inw W_t = W_t +\int_0^t \sigma_{\tau_u}^\top \rho_{\tau_u} du$ is a Brownian motion independent of $\tau$, and $\left(S_t/\beta_{\tau_t}(T-t)\right)_{t\in [0,T]}$ is a $\inw P$-martingale.
Moreover,  under $\inw P$, $(S,\tau)$ has infinitesimal generator
$\inw{\mathcal{H}}_{t}$ given by
\[
\inw{\mathcal{H}}_{t}f(s,i)=\tilde{\L}_{i}f(s,i)+\sum_{j=1}^{l}\left(\inw{\Lambda}_{T-t}\right)_{ij}f(s,j).
\]
\end{prop}
%\begin{RM}\label{rem:lawforward}
%${\inw{\mathcal{H}}}_{t}$ is the infinitesimal generator
%of a time inhomogeneous Markov process $(\inw{S},\inw{\tau})$,
%where the Markov chain $\inw{\tau}$ has infinitesimal generator
%$(\inw{\Lambda}_{T-t})$ and $\inw{S}$ is a regime-switching
%geometric Brownian motion with parameters $r_{i}\mathbf{1},a(i)$,
%$i\in\{1,\ldots,l\}$.
%\end{RM}

\subsection{Option price under $\tilde P$}
Suppose that the
payoff $\Phi$ is of
polynomial order\footnote{This means that $\Phi$ is bounded by a
constant times $\sum_{k=1}^n |p_k(s)|$, for some $n$, where each
$p_k$ is of the form $\prod_{j=1}^d
\left(s^{(j)}\right)^{\alpha_j}$, $\alpha_j \in \{0,1,2,\ldots\}$.}.
Under $\tilde P$ defines in Lemma \ref{lem:changeofmeasure0}, the value of a European option with payoff $\Phi$ at maturity $T$ is
\begin{equation}\label{eq:Cdefnew}
C_t(S_{t},\tau_{t}) =  E^{\tilde{P}}\left[ e^{-\int_{t}^{T}r_{\tau_{u}}du}\;\Phi(S_{T})|\calf_{t}\right] =  \frac{E\left[ B_T Z_{T} \;\Phi(S_{T})|\calf_{t}\right] }{B_t Z_{t}\gamma_{\tau_{t}}(T-t)}.
\end{equation}
Since $Z$ is a multiplicative functional by Lemma \ref{lem:changeofmeasure0}, $C$ can also be written as
\begin{equation}\label{eq:Crep1new}
C_t(s,i)=E_{s,i}\left[ B_{T-t}Z_{T-t} \Phi(S_{T-t})\right]
\Big{/}\gamma_{i}(T-t).
\end{equation}
The next result, proven in Appendix \ref{app:density},  shows that option prices are smooth.

\begin{prop}\label{prop:density}
The  density $f_{s,t}(y)$ of $S_t$ is infinitely differentiable with
respect to $s$ and $y$, and continuously differentiable with respect
to $t$ if ${\tilde \Lambda}_t$ is. Moreover all these derivatives are
integrable and the density possesses moments of all orders.
\end{prop}
Note that by Proposition \ref{prop:density}, $C_t(s,i)$ is continuously differentiable with respect to $t$,
and twice continuously differentiable with respect to $s \in
(0,\infty)^d$, for any $t <T$.
The next result, which follows from the definition of $C$ together with Lemma \ref{lem:changeofmeasure0} and Proposition \ref{prop:forward}, gives  two more
 representations for $C$.

\begin{lem}\label{lem:Crepnew}
Let  $C$ be given by \eqref{eq:Cdefnew}. If $(\tilde{S},\tilde{\tau})$ has infinitesimal generator $\tilde{\cH}_{t}$,
then
\begin{equation}
C(t,s,i)=E\left[ e^{-\int_{t}^{T}r_{\tilde{\tau}_{u}}du}\;\Phi(\tilde{S}_{T})|\tilde{S}_{t}=s,\tilde{\tau}_{t}=i\right] .\label{eq:Crep3new}
\end{equation}

Also, if $\left(\inw S,\inw{\tau}\right)$ is a process with generator
$\inw{\mathcal{H}}_{t}$, then
\begin{equation}\label{eq:Crep4new}
C(t,s,i)=\beta_{i}(T-t)E\left[ \Phi\left(\inw S_{T}\right)|\inw S_{t}=s,\inw{\tau}_{t}=i\right] .
\end{equation}
\end{lem}
\begin{RM}
Under the assumptions on $\Phi$, it follows from Proposition \ref{prop:density} that $C$ is smooth, and  an application of Ito's formula yields that $C$ satisfies
\begin{equation}\label{eq:odegammanew}
\partial_{t}C_{t}(s,i)+{\tilde{\cH}}_{t}C_{t}(s,i)=r_{i}C_{t}(s,i),\quad C_{T}(s,i)=\Phi(s).
\end{equation}
%However, this extension of the Black-Scholes equation will not be used here.
\end{RM}
%From this lemma, we deduce the following corollary using dominated
%convergence.
%
%\begin{cor}If $\Phi\notin C_{b}^{2}$ but is such that the expectations
%in (\ref{eq:Cdefnew}), (\ref{eq:Cdefnewbis}), (\ref{eq:Crep3new}),
%(\ref{eq:Crep4new}) are finite (we then say that $\Phi$ belongs
%to $\mathcal{C}_{E}$) and there exists a sequence $(\Phi_{n})_{n\geq0}$
%of $C_{b}^{2}$ converging a.s. to $\Phi$ and such that $\Phi_{n}\leq\Psi$
%for all $n$ with $\Psi\in\mathcal{C}_{E}$, then (\ref{eq:Cdefnew}),
%(\ref{eq:Cdefnewbis}), (\ref{eq:Crep3new}), (\ref{eq:Crep4new})
%are true.
%
%\end{cor}

\begin{RM} Note that using the Algorithm \ref{algo:MC} described in Appendix \ref{app:sim}, together with \eqref{eq:Crep4new},
it is easy to use a Monte-Carlo method to estimate $C_{t}$.
\end{RM}

%%%%%%%%%%%%%%%%%%%%%%%%%%%%%%%%%%%%%%%%%%%%%%%%
\subsection{Call option}\label{ssec:call1d}
%%%%%%%%%%%%%%%%%%%%%%%%%%%%%%%%%%%%%%%%%%%%%%%%

For the remaining of the section, suppose that $d=1$ and consider
the European call option payoff $\Phi(s)=\max(0,s-K)$. We show below that as in the Black-Scholes case,  one can write the price $C$ of the option using only the distribution function of the log-return $\log{S_T}$ under the forward measure $\inw P$ and another measure obtained by a Esscher transform. These distribution functions could be possibly computed by inverting their Laplace transforms or their characteristic functions. See Appendix \ref{app:Laplace} for the corresponding formulas.
First, we define the Esscher transform using the fact that
\begin{eqnarray*}
E^{\inw P}\left[ S_{T}|\calf_{t}\right]  & = & \frac{E^{\tilde{P}}\left[ e^{-\int_{t}^{T}r_{\tau_{u}}du}\; S_{T}|\calf_{t}\right] }{E^{\tilde{P}}\left[ e^{-\int_{t}^{T}r_{\tau_{u}}du}|\calf_{t}\right] }=\frac{S_{t}}{\beta_{\tau_{t}}(T-t)}.
\end{eqnarray*}
We introduce the change of measure by setting $\left.\frac{d\check{P}}{d\inw P}\right|_{\calf_{T}}=\frac{S_{T}}{E^{\inw P}\left(S_{T}\right)}$.
Next, according to \eqref{eq:Crep4new},
\begin{eqnarray}
C_{t}(s,i) & = & s\frac{E^{\inw P}\left[ S_{T}\I\left(S_{T}>K\right)|S_{t}=s,\tau_{t}=i\right] }{E^{\inw P}\left[ S_{T}|S_{t}=s,\tau_{t}=i\right] }
 -\beta_{i}(T-t)K{\inw P}\left[ S_{T}>K|S_{t}=s,{\tau}_{t}=i\right] \nonumber \\
 & = & s\check{P}\left[ S_{T}>K|S_{t}=s,\tau_{t}=i\right]-\beta_{i}(T-t)K{\inw P}\left[ S_{T}>K|S_{t}=s,{\tau}_{t}=i\right] \label{eq:call1d}.
 \end{eqnarray}
Furthermore, if  for a given $t$, $\inw f$ is the conditional density
of $\log\left(S_{T}/s\right)$ given $S_{t}=s,\tau_{t}=i$ under $\inw P$,
and if $\check{f}$ is the conditional density of $\log\left(S_{T}/s\right)$
given $S_{t}=s,\tau_{t}=i$ under $\check{P}$, then $\check{f}(x)=\beta_{i}(T-t)e^{x}\inw f(x)$,
so using \eqref{eq:call1d}, one finds
\begin{eqnarray*}
\partial_{s}C_{t}(s,i) & = & \check{P}\left[ S_{T}>K|S_{t}=s,\check{\tau}_{t}=i\right] +\check{f}\{\log(K/s)\}-\beta_{i}(T-t)\frac{K}{s}\inw f\{\log(K/s)\}\\
 & = & \check{P}\left[ S_{T}>K|S_{t}=s,\tau_{t}=i\right] .
\end{eqnarray*}

It only remains to find the distribution of $(S,\tau)$ under $\check P$. This is done in the next Lemma, whose proof is given in Appendix \ref{app:checkP}.

\begin{lem}\label{lem:checkP}
Under the change of measure $\left.\frac{d\check{P}}{d\inw P}\right|_{\calf_{T}}=\frac{S_{T}}{E^{\inw P}\left(S_{T}\right)}$,
$\left(S,\tau\right)$ has infinitesimal generator $\tilde{\mathcal{H}}_{t}$
defined by $\check{\mathcal{H}}_{t}f(s,i)=\check{\L}_{i}f(s,i)+\tilde{\Lambda}_{T-t}f(s,i)$,
where for any smooth $f$,
\[
\check{\L}_{i}f(s,i)=(r_{i}+a_{i})s\partial_{s}f(s,i)+\frac{a_{i}}{2}s^{2}\partial_{s}^{2}f(s,i), \quad i\in\{1,\ldots,l\}.
\]
\end{lem}

Having defined the price of the option under the equivalent martingale measure $\tilde P$, we are now in a position to find the optimal strategy $(\pi,\psi)\in \mathbb{R}\times \mathcal{A}$ minimizing the quadratic hedging error  $\mathit{HE}(\pi,\psi)$. This is the content of the next section.

%%%%%%%%%%%%%%%%%%%%%%%%%%%%%%%%%%%%%%%%%%%%%%%%%%%%%%%%%%%%%%%%%%%%%%%%%%%%%%%%
%%%%%%%%%%%%%%%%%%%%%%%%%%%%%%%%%%%%%%%%%%%%%%%%%%%%%%%%%%%%%%%%%%%%%%%%%%%%%%%%
\section{Optimal hedging strategy}\label{sec:hedging}
%%%%%%%%%%%%%%%%%%%%%%%%%%%%%%%%%%%%%%%%%%%%%%%%%%%%%%%%%%%%%%%%%%%%%%%%%%%%%%%%
%%%%%%%%%%%%%%%%%%%%%%%%%%%%%%%%%%%%%%%%%%%%%%%%%%%%%%%%%%%%%%%%%%%%%%%%%%%%%%%%

Let $C$ be defined by \eqref{eq:Crep1new}, and set
\begin{equation}\label{eq:alphadef}
\alpha_t(s,i)=\nabla_{s}C_{t}(s,i)+C_{t}(s,i)D^{-1}(s)\rho(i).
\end{equation}
%\begin{eqnarray} \alpha_t(,s,i) &=& D^{-1}(s)
%a(i)^{-1}D^{-1}(s)\left\{\mathcal{L}_i(C_t
%g)-g\mathcal{L}_i(C_t)-rgC_t\right\}(s,i) \nonumber\\
%&=& D^{-1}(s) a(i)^{-1}\left\{ m(i)C_t(s,i)
%+a(i)D(s)\nabla_s C_t(s,i) \right\} \nonumber\\
%&=& \nabla_s C_t(s,i)+ C_t(s,i) D^{-1}(s)\rho(i),
%\label{eq:alphadef}
%\end{eqnarray}
Further let $V$ be the solution of the stochastic differential equation
\begin{equation}\label{eq:Vdef}
V_{t}=C_0(s,i)+\int_{0}^{t}\alpha_u(S_{u},\tau_{u-})^{\top}dX_{u}-\int_{0}^{t}V_{u-}dM_{u},\quad (S_0,\tau_0)=(s,i),
\end{equation}
which exists and is unique, according to \citet[Theorem
V.7]{Protter:2004}.
The proof of the following lemma is given in Appendix \ref{app:alpha}.

\begin{lem}\label{lem:alpha} Define $G_{t}=B_{t}C_t(S_{t},\tau_{t})-V_{t}$, and set
\begin{eqnarray}
\phi_{t} & = & \alpha_t(S_{t},\tau_{t-})-V_{t-}D^{-1}(X_{t})\rho(\tau_{t-})\label{eq:phidef}\\
 & = & \nabla_{s}C_{t}(S_{t},\tau_{t-})+G_{t-}D^{-1}(X_{t})\rho(\tau_{t-}).\label{eq:phidef2}
\end{eqnarray} Then $\phi$ is predictable
and
\begin{equation}
V_{t}=C_{0}(S_{0},\tau_{0})+\int_{0}^{t}\phi_{u}^{\top}dX_{u},\label{eq:Vdef2}
\end{equation}
\end{lem}
Equation \eqref{eq:Vdef2} shows that $V_{t}$ is the discounted
value at time $t$ of a portfolio with initial value $C_0(s,i)$ and
investment strategy $\phi$, so $G_{t}$ is the corresponding hedging
error at this period.
Recall that  $C_0 = E^{\tilde P}[B_T \Phi(S_T)]$ is a real option price, corresponding to the price under the variance-optimal  martingale measure $\tilde P$. In other words, if all agents in the economy
 minimize a mean-variance criterion for the hedging error under a RSGBM model, then the pricing kernel will be generated by the variance-optimal measure $\tilde P$.

It is interesting to note that
$\phi_{0}=\nabla_{s}C_{0}(s,i)$, but in general,
$\phi_{t}\neq\nabla_{s}C_{t}(S_{t},\tau_{t-})$, unless $G_{t-}=0$,
indicating perfect hedging.

\begin{RM}\label{rem:alphaest} Since
$\alpha_t(s,i)=\nabla_{s}C_{t}(s,i)+C_{t}(s,i)D^{-1}(s)\rho(i)$,
$i\in\{1,\ldots,l\}$, one can obtain an unbiased estimate of
$\alpha_{t}$ through simulations by using  an obvious extension of
the ``pathwise method'' in \citet{Broadie/Glasserman:1996}. More
precisely, if $\Phi$ is differentiable almost everywhere, then
\eqref{eq:Crep4new} yields
\begin{equation}\label{eq:deltaRSGBM}
\nabla_{s}C_{t}(s,i)=\beta_{i}(T-t)D^{-1}(s)E^{\inw P}\left[ D(S_{T})\nabla\Phi(S_{T})|S_{t}=s,\tau_{t}=i\right] ,
\end{equation}
so $\alpha_{t}$ is an expectation of a function of $S_{T}$. It follows
from \eqref{eq:phidef} and \eqref{eq:deltaRSGBM} that $\phi_{t}$
can also be estimated by Monte-Carlo methods.
\end{RM}
Next, we find an expression
for the hedging error $G$ of the investment strategy $(C_{0},\phi)$. It's proof is given in Appendix \ref{app:pfGrep}.
\begin{lem}\label{lem:Grep} The hedging error $G$ satisfies
\begin{equation}
G_{t}=-\int_{0}^{t}G_{u-}dM_{u}+\sum_{0<u\le t}B_u \Delta C_{u}(S_{u},\tau_{u})-\int_{0}^{t}B_{u}\tilde{\Lambda}_{T-u}C_{u}(S_{u},\tau_{u})du.\label{eq:Grep}
\end{equation}
\end{lem}
The next result,  proven in  Appendix \ref{app:pfGSmart},  is the key for solving the hedging  problem.

\begin{lem}\label{lem:GSmart} For all $0\le t\le T$, $\gamma_{\tau_{t}}(T-t)G_{t}$
and $\gamma_{\tau_{t}}(T-t)X_{t}G_{t}$ are martingales. In
particular $E(G_{T})=0$, and for any stopping times $U,V$ with $0\le
U\le V\le T$,
\begin{equation}
E\left\{ G_{T}(X_{V}-X_{U})|\calf_{U}\right)=0.\label{eq:optcont}
\end{equation}
\end{lem}

One can now state the other main result of the paper, proven in Appendix
\ref{pf:optcont}. Recall that the quadratic hedging error $\mathit{HE}$ is defined in \eqref{eq:problem}.

\begin{thm}\label{thm:optcont} The optimal solution of the quadratic
hedging problem for a regime-switching geometric Brownian motion is
given by $(C_{0},\phi)$, as defined by equation \eqref{eq:phidef}
and the actualized value of the associated portfolio satisfies \eqref{eq:Vdef}.
More precisely, for any $(\pi,\psi)\in\R\times\mathcal{A}$,
\begin{equation}\label{eq:optsol}
\mathit{HE}(\pi,\psi)  \ge  \mathit{HE}(C_{0},\phi),
\end{equation}
and $\mathit{HE}(\pi,\psi)  = \mathit{HE}(C_{0},\phi)$ if and only if $\pi=C_{0}$ and $\psi=\phi$ a.s. In particular, the solution to the optimization problem exists and is unique.
\end{thm}

\subsection{Particular cases}

\subsubsection{Geometric Brownian motion}

In this case, $\Lambda\equiv0$ and $C$ does  not depend on $\tau$, so by Lemma \ref{lem:Grep},
$G_{t}=-\int_{0}^{t}G_{u}dM_{u}$, with $G_{0}=0$. Since the solution
of \eqref{eq:Grep} is unique, $G\equiv0$,
proving the perfect hedging, as it is well-known for the
Black-Scholes model. Also, using \eqref{eq:phidef2}, one recovers
the  delta-hedging value $\phi_{t}=\nabla_{s}C_{t}(S_{t})$.

\subsubsection{Risk neutral measure}

To recover known results from the literature, suppose that $X_{t}=B_{t}S_{t}$
is a martingale. It then follows from \eqref{eq:Xmart} that $m\equiv0$,
so $\rho\equiv0$, $\ell\equiv0$, $\gamma\equiv1$, ${\tilde{\mathcal{H}}}_{t}=\cH$,
$M\equiv0$, and $Z\equiv1$. Next, we get from \eqref{eq:Crep3new},
\eqref{eq:alphadef} -- \eqref{eq:phidef} that
\begin{eqnarray*}
C_{t}(s,i) & = & E\left\{ B_{T-t}\Phi(S_{T-t})|S_{0}=s,\tau_{0}=i\right\} ,\\
\alpha_t(,s,i) & = & \nabla C_{t}(s,i),\\
V_{t} & = & C_0(s,i)+\int_{0}^{t}\alpha_u(S_{u},\tau_{u-})^{\top}dX_{u},\\
\phi_{t} & = & \alpha_t(S_{t-},\tau_{t-}).
\end{eqnarray*}

%%%%%%%%%%%%%%%%%%%%%%%%%%%%%%%%%%%%%%%%%%%%%%%%%%%%%%%%%%%%%%%%%%%%%%%%%%%%%%%%
%%%%%%%%%%%%%%%%%%%%%%%%%%%%%%%%%%%%%%%%%%%%%%%%%%%%%%%%%%%%%%%%%%%%%%%%%%%%%%%%
\section{Examples of application}\label{sec:application}
%%%%%%%%%%%%%%%%%%%%%%%%%%%%%%%%%%%%%%%%%%%%%%%%%%%%%%%%%%%%%%%%%%%%%%%%%%%%%%%%
%%%%%%%%%%%%%%%%%%%%%%%%%%%%%%%%%%%%%%%%%%%%%%%%%%%%%%%%%%%%%%%%%%%%%%%%%%%%%%%%

In this section,  we present two implementations of our proposed
methodology, using formulas \eqref{eq:Crep4new} and
\eqref{eq:deltaRSGBM}  together with Monte Carlo simulations to
estimate respectively the value of an option and the initial
investment in the risky asset. The simulations are based on Algorithm \ref{algo:MC} described in Appendix \ref{app:sim}.

%%%%%%%%%%%%%%%%%%%%%%%%%%%%%%%%%%%%%%%%%%%%%%%%%%%%%%%%%%%%%%%%%%%%%%%%%%%%%%%%
\subsection{A first example}\label{ssec:shenetal}
%%%%%%%%%%%%%%%%%%%%%%%%%%%%%%%%%%%%%%%%%%%%%%%%%%%%%%%%%%%%%%%%%%%%%%%%%%%%%%%%

In \citet{Shen/Fan/Siu:2014}, the authors consider a RSGBM with parameters
$\mu=(0.04,0.08)^{\top}$, $\sigma=(0.4,0.2)^{\top}$, $r=(0.02,0.04)^{\top}$
and generator $\Lambda=0.5\left(\begin{array}{rr}
-1 & 1\\
1 & -1
\end{array}\right)$.
Also, $S_{0}=100$, $T=1$. Note that under their risk neutral measure, the regime-switching process $\tau$ is still time-homogeneous.
In order to be able to make some comparisons with our model, only their SRS model is
chosen. Using inverse Fourier transforms, they obtained the value of call options given in Table \ref{tab:callShenetall}.

%\begin{table}[h!]
%\caption{Call option prices computed with a discretization of $2^{12}$ points
%for the inverse Fourier transform.}
%
%
%\label{tab:callShenetall}
%
%\centering{}%
%\begin{tabular}{|l|c|c|}
%\hline
% & \multicolumn{2}{c|}{ RSGBM model}\tabularnewline
%Strike  & Regime 1  & Regime 2\tabularnewline
%\hline
%70   & 34.0904 (0.86\%)  & 33.1151 (1.07\%) \tabularnewline
%80   & 26.7779 (1.43\%)  & 24.5557 (2.20\%) \tabularnewline
%90   & 20.6144 (1.96\%)  & 17.1617 (3.50\%) \tabularnewline
%100  & 15.6171 (2.27\%)  & 11.3358 (4.42\%) \tabularnewline
%110  & 11.6953 (2.21\%)  & 7.1553  (4.28\%) \tabularnewline
%120  & 8.6931  (1.68\%)  & 4.3873  (2.35\%) \tabularnewline
%\hline
%\end{tabular}
%\end{table}

\begin{table}[h!]
\caption{Call option prices computed with a discretization of $2^{12}$ points
for the inverse Fourier transform \citep{Shen/Fan/Siu:2014}.}\label{tab:callShenetall}

\centering{}%
\begin{tabular}{|l|c|c|c|c|c|c|}
\hline
Strike & 70   & 80   & 90   & 100  &  110  &  120  \\
\hline
Regime 1 & 34.0904   & 26.7779    & 20.6144  & 15.6171      & 11.6953   & 8.6931     \\
\hline
Regime 2 & 33.1151   & 24.5557   & 17.1617    & 11.3358    & 7.1553      &  4.3873\\
\hline
 \end{tabular}
\end{table}

In order to implement our proposed methodology using Monte
Carlo simulations,  one has to find $\lambda =
\max_{i\in\{1,2\}}\sup_{t\in [0,1]} -\left(\inw
\Lambda_t\right)_{ii}$, as described in Algorithm \ref{algo:MC}.
Based on panel (a) in Figure \ref{fig:lambdatShen}, one can take
$\lambda=.5185$. Note also that the discounting factors
$\beta_{i}(T)$ are given by $0.9767$ for regime 1 and $0.9644$ for
regime 2. Their values are displayed in panel (b) of Figure
\ref{fig:lambdatShen}. Next, in Table \ref{tab:callMC}, 95\%
confidence intervals are computed for the call values and the
initial investment $\phi_0$ (number of shares), based on  $10^6$
simulated values of $\inw S_T$. Although the model proposed by \citet{Shen/Fan/Siu:2014} and our model are different under the
martingale measures, the option values obtained from both models are comparable.
However, one can see that under regime 1, our results are
significantly lower than those of \citet{Shen/Fan/Siu:2014}, while they are significantly larger for regime
2. Estimated call values and $\phi_0$ for the two regimes are
displayed in Figure \ref{fig:evalshen} for strike ranging
from $50$ to $150$.

\begin{figure}[h!]
\begin{centering}
(a)\includegraphics[scale=0.4]{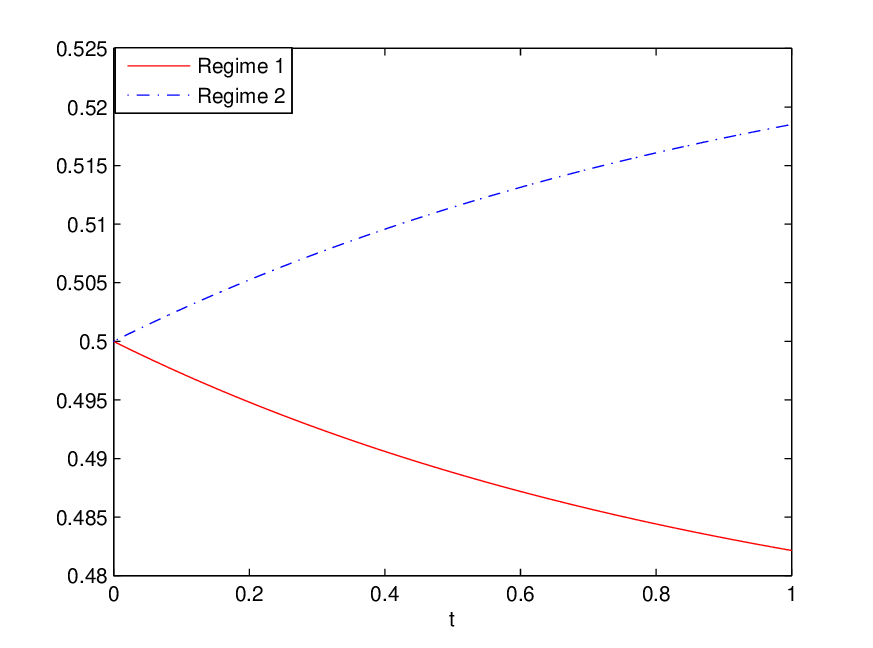}
(b)\includegraphics[scale=0.4]{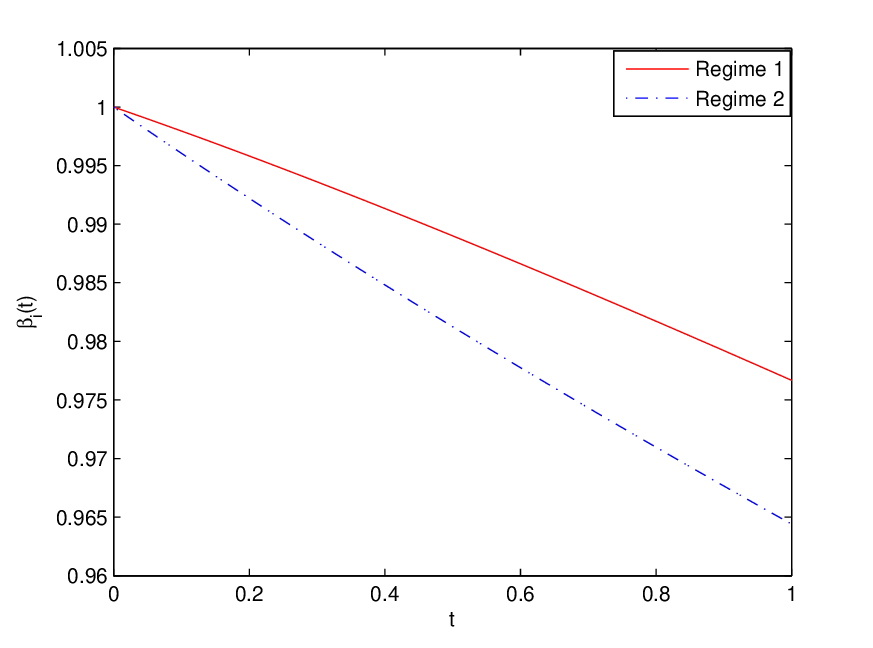}
\par\end{centering}
\caption{Graphs of $-\left(\inw{\Lambda}_{t}\right)_{ii}$ (panel a)
and $\beta_{i}(t)$ (panel b), for $i\in\{1,2\}$.}
\label{fig:lambdatShen}
\end{figure}

\begin{table}[h!]
\caption{Call option prices and initial investment values computed  by Monte Carlo
with $10^{6}$ simulations, using antithetic variables.}

\label{tab:callMC}

\begin{tabular}{|l|c|c|c|c|}
\hline
 & \multicolumn{2}{c|}{ Call value} & \multicolumn{2}{c|}{Initial value $\phi_{0}$}\tabularnewline
Strike  & Regime 1  & Regime 2  & Regime 1  & Regime 2\tabularnewline
\hline
 70  & $ 34.0580  \pm 0.0489$ &  $33.1345   \pm  0.0347 $  &  $ 0.8935  \pm   0.0007 $ & $  0.9547   \pm 0.0005  $  \\
 80  & $ 26.6933  \pm 0.0460$ &  $24.6400   \pm  0.0328 $  &  $ 0.8112  \pm   0.0008 $ & $  0.8863   \pm 0.0006  $  \\
 90  & $ 20.4806  \pm 0.0424$ &  $17.3062   \pm  0.0299 $  &  $ 0.7103  \pm   0.0009 $ & $  0.7672   \pm 0.0008  $  \\
100  & $ 15.4499  \pm 0.0384$ &  $11.5217   \pm  0.0262 $  &  $ 0.6005  \pm   0.0010 $ & $  0.6111   \pm 0.0008  $ \\
110  & $ 11.5158  \pm 0.0343$ &  $ \phantom{1} 7.3565   \pm  0.0222 $  &  $ 0.4931  \pm   0.0010 $ & $  0.4492   \pm 0.0009  $ \\
120  & $ \phantom{1} 8.5192  \pm 0.0303$ &  $ \phantom{1} 4.5829   \pm  0.0184 $  &  $ 0.3958  \pm   0.0009 $ & $  0.3095   \pm 0.0008  $ \\
\hline
\end{tabular}
\end{table}

\begin{figure}[h!]
\begin{centering}
(a)\includegraphics[scale=0.40]{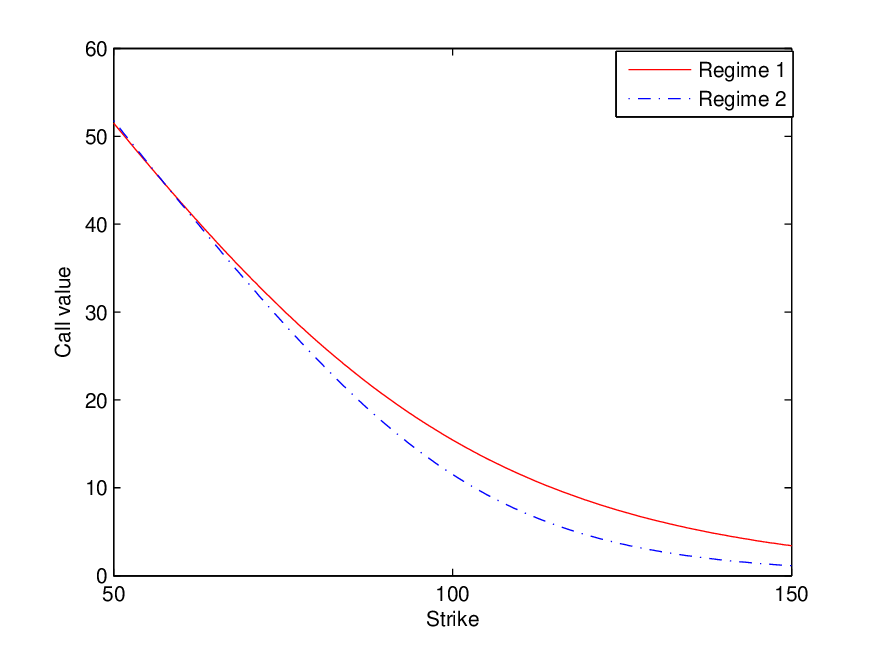}
(b)\includegraphics[scale=0.40]{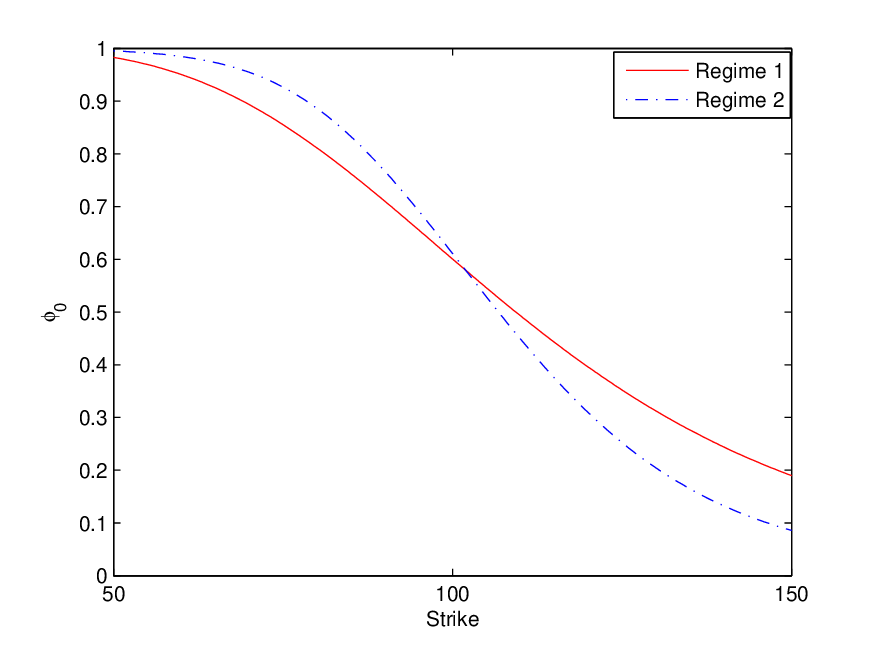}
\par\end{centering}

\caption{Graphs of $C(0,100,i)$ (panel a) and $\phi_0(i)$ (panel b)
for $i\in\{1,2\}$ and for strikes in $[50,150]$.}
\label{fig:evalshen}
\end{figure}

\begin{RM}
A clear advantage of our methodology over the one proposed by
\citet{Shen/Fan/Siu:2014} is that we have  an interesting
justification for the choice of the equivalent martingale measure.
Furthermore, our methodology is easy to implement for any number of
underlying assets while a discretization used to compute inverse
Fourier transforms will be almost impossible to implement for more
than three assets.
\end{RM}

\subsection{Second example}\label{ssec:ex-apple}
Next, we consider the log-returns of the stock price
of Apple (aapl), from January 2nd, 2010 to April 24th, 2015. In
order to estimate the parameters of the model,  we first fit a
regime-switching Gaussian random walk model on the returns,
following the methodology proposed in \citet[Chapter 10]{Remillard:2013}. In this case,
one has a discrete time Markov chain for the regimes, instead of a
continuous time one. Also, given a regime $i$ at period $t$, the
return at this period is Gaussian with a mean and variance depending
on the regime. It is easy to check that the continuous time limit in
this case is a regime-switching Brownian motion. According to
\citet{Remillard:2013} and
\citet{Remillard/Hocquard/Lamarre/Papageorgiou:2014}, the
appropriate number of regimes should be the first number so that one
does not reject the null hypothesis of a Gaussian regime-switching.
In the present case, the P-value for only one regime is 0\%, while
the P-value is $8\%$ for two regimes (using 1000 bootstrap samples).
The estimated parameters are given in Tables \ref{tab:param1-apple},
while the transition matrix is $Q = \left(
\begin{array}{cc}
 0.7600  &  0.2400\\
 0.0590 & 0.9410
 \end{array}\right)$.

 \begin{table}[h!]
\caption{Parameters estimation for the two-regime model on Apple
daily returns from 01/02/2010 to 04/23/2015. }
\label{tab:param1-apple}

\centering{}%
\begin{tabular}{c|c|c|c|c}
Regime  & Mean  & Volatility  & Stat. distr.  & Prob. of regime at
current time\tabularnewline \hline 1  & -0.0018   &  0.0283  &
0.1973 & 0.0441 \tabularnewline 2  &  0.0018   &  0.0123  & 0.8027
& 0.9559 \tabularnewline
\end{tabular}
\end{table}
To obtain the associated parameters of a RSGBM on an annual time
scale, one can solve $Q = \exp\{\Lambda/252\}$. Here, one gets $\Lambda
= \left(
\begin{array}{cc}
-71.8620 &   \phantom{-} 71.8620  \\
\phantom{-} 17.6661 &  -17.6661
 \end{array}\right)$.
 \begin{RM} Solving $Q = \exp\{\Lambda/252\}$ is not always
 possible. In practice,
 $\Lambda \approx 252(Q-I)$ is used. However, the relationship $Q =
 \exp\{\Lambda/252\}$  no longer holds and their could be  discrepancies for option values in the discrete and continuous cases.
\end{RM}
To find the remaining parameters in continuous time (measured in
years), one can  use the fact that for regime $i$, the mean of the
regime-switching random walk is approximately
$(\mu_i-\sigma_i^2/2)/252$, while its variance is  approximately
$\sigma_i^2/252$. The resulting parameters  are given in Table
\ref{tab:paramcont1}. One can then interpret regime 1 as a ``bad''
regime, since $\mu_1<0$ is negative and $\sigma_1$ is  twice
the value in regime 2, and $\mu_2>0$.
Fortunately, it follows from the stationary distribution that the
``good'' regime appears 80\% of the time. Also, there is a 95.5\%
chance that the current regime on April 24th, 2015, is the good one.
Finally, the estimated (annual) volatility is $.2658$, and the
annual rate, based on the 1-month Libor, is 2.16\%.
\begin{table}[h!]
\caption{Parameters for the RSGBM.} \label{tab:paramcont1}

\centering{}%
\begin{tabular}{c|c|c}
Regime  & $\mu$  & $\sigma$  \tabularnewline
\hline
1  &  -0.3436  & 0.4486  \tabularnewline
2  &   0.4813  & 0.1945  \tabularnewline
\end{tabular}
\end{table}
In order to compare some models, one  will use the values of Table
\ref{tab:strikesapple}, giving the bid/ask  of call options on
Apple, expiring at $T=20/252$, for three strike prices. As one can
see from Table \ref{tab:strikesapple}, the Black-Scholes model
under-estimate the market values.
\begin{table}[h!]
\caption{Market values, Black-Scholes price (BS) and delta-hedging
value ($\phi_0$) for a call on Apple, as of April 24th 2015. The
maturity day is May 22nd 2015 and the observed stock value is
$129.95$. } \label{tab:strikesapple}
\centering{}%
\begin{tabular}{|l|c|c|c|c|}
\hline
Strike & Bid  & Ask  & BS & $\phi_0$\\
\hline
$128$ & $5.50$ & $5.65$ & $5.0304 $ & $0.6034$    \\
$129$ & $4.95$ & $5.10$ & $4.4776 $ & $0.5629$   \\
$130$ & $4.45$ & $4.50$ & $3.9658 $ & $0.5220$   \\
\hline
\end{tabular}
\end{table}
Next, as a first approximation, we can compute the call option
prices and initial investment values using the semi-exact method for
optimal discrete time hedging, as described in Chapter 3 of
\citet{Remillard:2013}. See also \citet{Remillard/Rubenthaler:2013}.
The values corresponding to a daily hedging are given in Table
\ref{tab:callappleHMM}. It is worth noting that the Black-Scholes
prices are larger than the regime 2 prices, and smaller than the
regime 1 values.
\begin{table}[h!]
\caption{Call option prices and initial investment values using the
semi-exact method optimal hedging method as in Chapter 3 of
\citet{Remillard:2013}. The number of hedging periods is 20 and the
grid is composed of 2000 equidistant points in the interval
$[50,200]$.}\label{tab:callappleHMM}

\begin{tabular}{|l|c|c|c|c|}
\hline
 & \multicolumn{2}{c|}{ Call value} & \multicolumn{2}{c|}{Initial value $\phi_{0}$}\tabularnewline
Strike  & Regime 1  & Regime 2  & Regime 1  & Regime 2   \tabularnewline
\hline
$ 128 $ & $ 5.3981 $ &  $ 4.9069  $ &  $ 0.5951$ & $ 0.6166  $   \\
$ 129 $ & $ 4.8430 $ &  $ 4.3385  $ &  $ 0.5578$ & $ 0.5748  $   \\
$ 130 $ & $ 4.3268 $ &  $ 3.8141  $ &  $ 0.5202$ & $ 0.5321  $   \\
\hline
\end{tabular}
\end{table}
Finally, we can find the option values provided by the optimal
hedging for the RSGBM. Since the interest rate is not assumed to
depend on the regime in this example, one gets $\tilde \Lambda =
\inw\Lambda$. Its graph is displayed in Figure
\ref{fig:lambdatapple}. From it, one see that one can take
$\lambda=72.2522$. The 95\% confidence intervals for the call value
and the initial number of shares $\phi_0$, based on $10^6$
simulations of $\tilde S_T$,  are given in Table
\ref{tab:callMCapple}. The resulting values for strike prices
ranging from $125$ to $150$ are displayed in Figure
\ref{fig:callapple}.
Comparing Tables \ref{tab:callappleHMM} and \ref{tab:callMCapple},
it seems that as the number of hedging periods tends to infinity,
the option value under regime 1 decreases while it increases under
regime 2. Also, the Black-Scholes values appear significantly larger
than those of regime 1 and regime 2.

\begin{figure}[h!]
\begin{centering}
\includegraphics[scale=0.40]{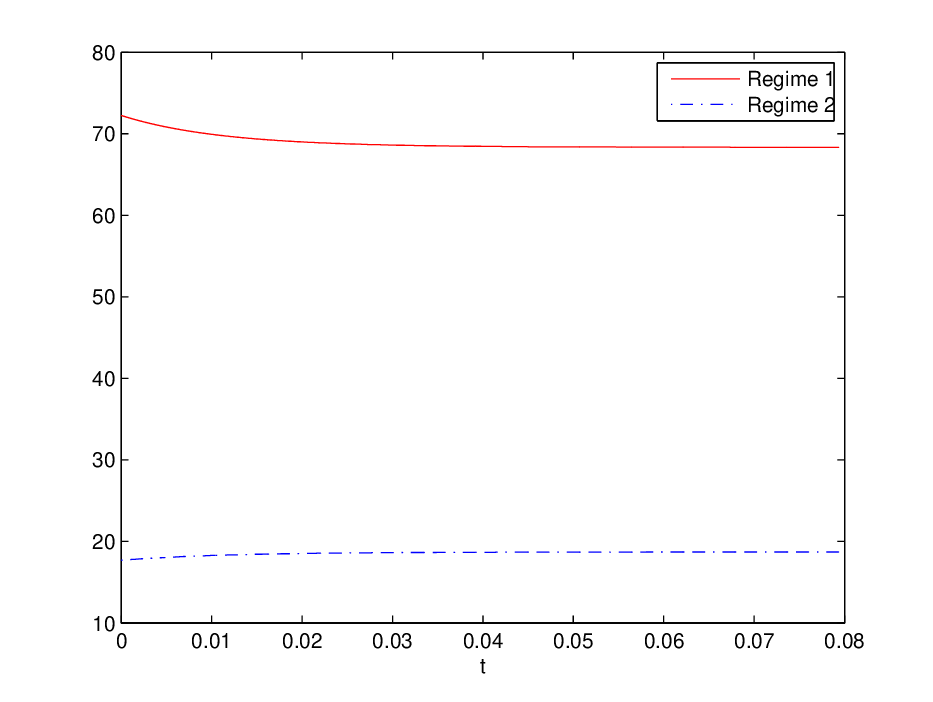}
\end{centering}

\caption{Graph of $-\left(\tilde{\Lambda}_{t}\right)_{ii}$ for
$i\in\{1,2\}$ and $t\in[0,1/12]$.} \label{fig:lambdatapple}
\end{figure}

\begin{table}[h!]
\caption{Call option prices and initial investment values computed  by Monte Carlo
with $10^{6}$ simulations, using antithetic variables.}\label{tab:callMCapple}

\begin{tabular}{|l|c|c|c|c|c|c|}
\hline
 & \multicolumn{2}{c|}{ Call value} & \multicolumn{2}{c|}{Initial value $\phi_{0}$}\tabularnewline
Strike  &  Regime 1  & Regime 2   & Regime 1  & Regime 2\tabularnewline
\hline
$ 128$  & $ 5.0210    \pm 0.0094 $ &  $ 4.9813  \pm  0.0094  $&  $0.6070   \pm  0.0007  $ & $ 0.6092   \pm 0.0007   $  \\
$ 129$  & $ 4.4653    \pm 0.0090 $ &  $ 4.4236  \pm  0.0090  $&  $0.5648   \pm  0.0007  $ & $ 0.5659   \pm 0.0007   $  \\
$ 130$  & $ 3.9523    \pm 0.0085 $ &  $ 3.9097  \pm  0.0085  $&  $0.5222   \pm  0.0007  $ & $ 0.5222   \pm 0.0007   $  \\
\hline
\end{tabular}
\end{table}

\begin{figure}[h!]
\begin{centering}
(a)\includegraphics[scale=0.40]{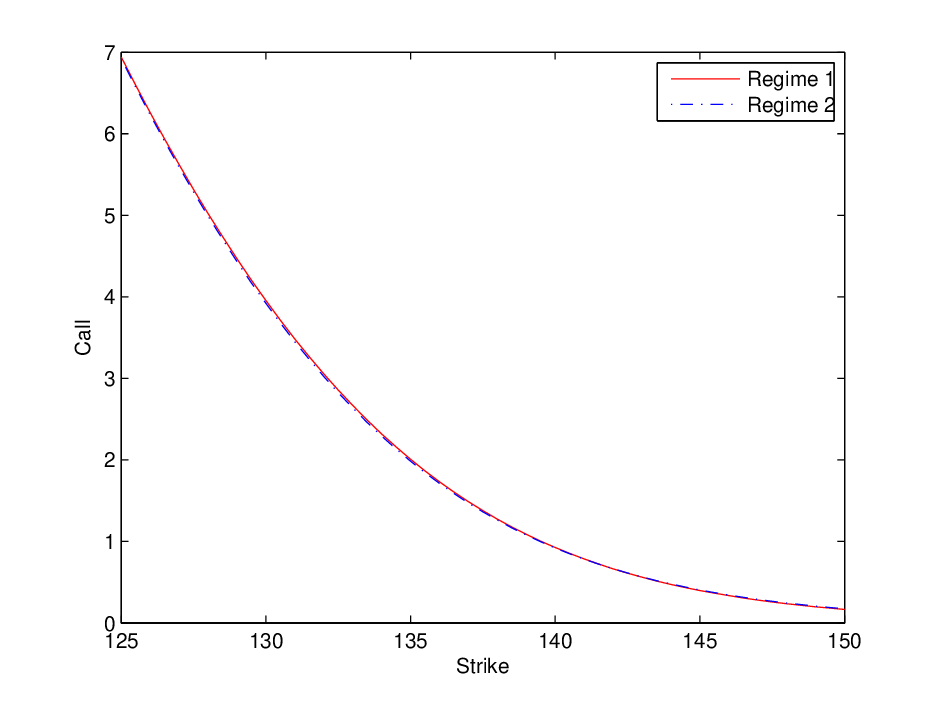}
(b)\includegraphics[scale=0.40]{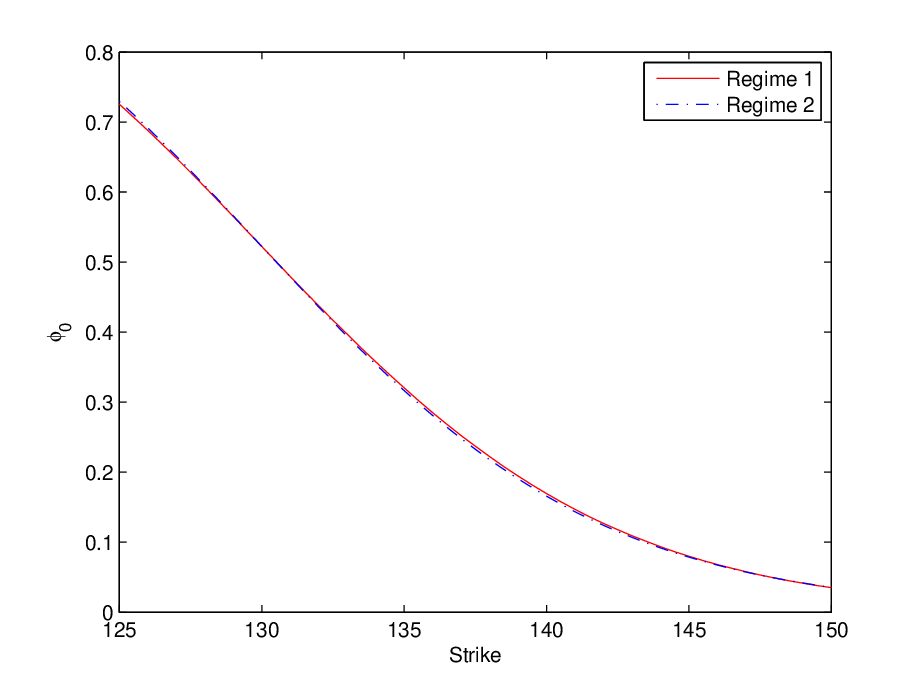}
\par\end{centering}
\caption{Graphs of $C(0,100,i)$ (panel a) and $\phi_0(i)$ (panel b)
for $i\in\{1,2\}$ and for strike values in $[125,150]$.}
\label{fig:callapple}
\end{figure}
Note that under-estimating the option value is not necessarily a bad
thing.  As proposed in \citet{Remillard:2013}, since the market
values are all larger, one could short the call option, invest $C_0$
(as computed in Table \ref{tab:callMCapple}) in a portfolio based on
optimal hedging for
 the RSGBM model, and invest the remaining value in a risk-free account. If
transactions fees were negligible, one could make on average the
difference between the initial option prices. This  was
tested on real data in
\citet{Remillard/Hocquard/Lamarre/Papageorgiou:2014}.
%%%%%%%%%%%%%%%%%%%%%%%%%%%%%%%%%%%%%%%%%%%%%%%%%%%%%%%%%%%%%%%%%%%%%%%%%%%%%%%%
%%%%%%%%%%%%%%%%%%%%%%%%%%%%%%%%%%%%%%%%%%%%%%%%%%%%%%%%%%%%%%%%%%%%%%%%%%%%%%%%

\section{Conclusion}

%%%%%%%%%%%%%%%%%%%%%%%%%%%%%%%%%%%%%%%%%%%%%%%%%%%%%%%%%%%%%%%%%%%%%%%%%%%%%%%%
%%%%%%%%%%%%%%%%%%%%%%%%%%%%%%%%%%%%%%%%%%%%%%%%%%%%%%%%%%%%%%%%%%%%%%%%%%%%%%%%

In this paper we presented the variance-optimal equivalent martingale measure $\tilde P$ and we found the optimal hedging solution for pricing
and hedging an option under a regime-switching geometric Brownian
motion model,  extending the Black-Scholes model. The option price
and  the optimal strategy can be deduced from $\tilde P$,  whose law is easy to simulate.  Compared to some
existing methods relying on inverse Fourier transforms, the proposed
methodology has the advantage that it is easy to implement for any
number of underlying assets.

%%%%%%%%%%%%%%%%%%%%%%%%%%%%%%%%%%%%%%%%%%%%%%%%%%%%%%%%%%%%%%%%%%
%\bibliography{refBR}
\bibliographystyle{apalike}

%%%%%%%%%%%%%%%%%%%%%%%%%%%%%%%%%%%%%%%%%%%%%%%%%%%%%%%%%%%%%%%%%%

%%%%%%%%%%%%%%%%%%%%%%%%%%%%%%%%%%%%%%%%%%%%%%%%%%%%%%%%%%%%%%%%%%%%%%%%%%%%%%%%
%%%%%%%%%%%%%%%%%%%%%%%%%%%%%%%%%%%%%%%%%%%%%%%%%%%%%%%%%%%%%%%%%%%%%%%%%%%%%%%%

\appendix
%dummy comment inserted by tex2lyx to ensure that this paragraph is not empty
%dummy comment inserted by tex2lyx to ensure that this paragraph is not empty
%%%%%%%%%%%%%%%%%%%%%%%%%%%%%%%%%%%%%%%%%%%%%%%%%%%%%%%%%%%%%%%%%%%%%%%%%%%%%%%%
%%%%%%%%%%%%%%%%%%%%%%%%%%%%%%%%%%%%%%%%%%%%%%%%%%%%%%%%%%%%%%%%%%%%%%%%%%%%%%%%

%%%%%%%%%%%%%%%%%%%%%%%%%%%%%%%%%%%%%%
%%%%%%%%%%%%%%%%%%%%%%%%%%%%%%%%%%%%%%
\section{Construction of inhomogeneous Markov chains}
\label{app:sim}
%%%%%%%%%%%%%%%%%%%%%%%%%%%%%%%%%%%%%%
%%%%%%%%%%%%%%%%%%%%%%%%%%%%%%%%%%%%%%

Suppose that the (inhomogeneous) Markov chain $\tau_{t}$ has
infinitesimal generator $\Lambda_{t}$, i.e., for any $j\neq i$,
$\displaystyle
\lim_{h\downarrow0}\frac{1}{h}P(\tau_{t+h}=j|\tau_{t}=i)=(\Lambda_{t})_{ij}$.
For a given $T$, assume that one can find $0<\lambda$ so that
$-(\Lambda_{t})_{ii}\le\lambda$, for any $i\in \{1,\ldots,l\}$ and
any $t\in [0,T]$.
\begin{algo}\label{algo:MC}
To construct the Markov chain $\tau$, do the following:
\begin{itemize}
\item Generate $N\sim{\rm Poisson}(\lambda T)$.
\item If $N=n$, generate independent uniform variates $U_{1},\ldots,U_{n}$
and order them. Denote the resulting sample by
$U_{n:1},\ldots,U_{n:n}$, with $U_{n:1}<U_{n:2}\cdots<U_{n:n}$. This
can be done by generating $n+1$ independent exponential variates
$E_{1},\ldots,E_{n+1}$ and by setting
$U_{n:i}=\frac{\sum_{j=1}^{i}E_{j}}{\sum_{j=1}^{n+1}E_{j}}$,
$i\in\{1,\ldots,n\}$.
\item Set $t_{i}=T\times U_{n:i}$, $i=1,\ldots,n$. These values are the
possible transition times. Further set $t_{0}=0$ and $t_{n+1}=T$.
\item For $k=1,\ldots,n$, if $\tau_{t_{k-1}}=i$, then $\tau_{t_{k}}=j$
with probability $P_{k,ij}$, where
$
P_{k,ij}=(\Lambda_{t_{k}})_{ij}/\lambda$ for $j\neq i$, and
$P_{k,ii}=1+(\Lambda_{t_{k}})_{ii}/\lambda$.
\end{itemize}
\end{algo}

\subsection{Proof of Proposition \ref{prop:density}}\label{app:density}
To show that the density of
$S_t$ is infinitely differentiable with respect to $s$ and $t$, we use Algorithm \ref{algo:MC}, together with representation \eqref{eq:sdesol}.
Conditional on $N_t=n$ and $\mathcal{F}_T^\tau$,  the returns $R_t^{(i)} =
\log\left(S_t^{(i)}/s^{(i)}\right)$, $i\in \{1,\ldots,d\}$, are
Gaussian, with mean $tV_t$, with $ V_t = \frac{1}{t}\int_0^t
v_{\tau_u} du $ and covariance matrix $tA_t$, with $A_t =
\frac{1}{t}\int_0^t a_{\tau_u}du $, where $ V_t = \sum_{k=1}^n
\left(U_{n:k}-U_{n:k-1}\right) v(\tau_{t_{k-1}})+
\left(1-U_{n:n}\right)v(\tau_{t_n})$  and
$
A_t = \sum_{k=1}^n  \left(U_{n:k}-U_{n:k-1}\right)a(\tau_{t_{k-1}})+
(1-U_{n:n})a(\tau_{t_n})$. Here, $U_{n:0}=0$.
 Next, set if $\alpha_0=i$ and $u_0=0$.  Then, for any
bounded measurable functions $\phi, \psi$, given $N_t=n$, one gets
\begin{eqnarray*}
E\{\phi(V_t)\psi(A_t)\}&=& n!  \sum_{\alpha_1=1}^l \cdots
\sum_{\alpha_n=1}^l \int_{0<u_1< \cdots <u_n < 1} \prod_{k=1}^n
g_{\alpha_{k-1}
\alpha_k}\left(\frac{\Lambda_{t u_k}}{\lambda}\right)\\
&& \quad  \times \phi\left\{\sum_{k=1}^n
(u_k-u_{k-1})v(\alpha_{k-1})+ (1-u_n)v(\alpha_{n})\right\}\\
&& \qquad  \times \psi\left\{\sum_{k=1}^n
(u_k-u_{k-1})a(\alpha_{k-1})+ (1-u_n)a(\alpha_{n})\right\}du_1
\cdots d u_n,
\end{eqnarray*}
where $g_{ij}(x) = x_{ij}$ when $j\neq i$, and $g_{ii}(x) =
1+x_{ii}$, $i\in \{1,\ldots,l\}$. Hence, if $\Lambda$ is
continuously differentiable with respect to $t$, then it follows
that the density of $S_t$ is infinitely differentiable with respect
to $s$, as well as continuously differentiable with respect to $t$.
Moreover these derivatives are all integrable.
\qed

\section{Auxiliary results}

%The following result is proved in \cite[Theorem 28, page
%75]{Protter:2004}.
%\begin{thm}\label{thm:qv2}
%Let $X$ and $Y$ be two semimartingales, and let $H,K \in
%\mathbb{L}$. Then
%$$
%[H\cdot X,K\cdot Y]_t = \int_0^t H_u K_u d[X,Y]_u.
%$$
%\end{thm}
%The following result is proved in \cite[Theorem 37, page
%84]{Protter:2004}.
%\begin{thm}\label{thm:dd}
%Let $X$ be a semimartingale with $X_0=0$. Then there exists a
%(unique) semimartingale $Z$ that satisfies the equation
%$Z_t=1+\int_0^t Z_{u-}dX_u$. $Z$ is given by
%$$
%Z_t = \E_t(X) = \exp\left\{X_t- \frac{1}{2}[X,X]_t\right\} \prod_{0
%< u \le t}(1+\Delta X_u)\exp\left\{-\Delta X_u+ \frac{1}{2}(\Delta
%X_u)^2\right\}
%$$
%where the infinite product converges.
%\end{thm}
%The following result is proved in \citet{Protter:2004}.
%\begin{thm}[It\^o's formula]\label{thm:ito}
%\begin{eqnarray*}
%f(t,X_t) &=& f(0,X_0) + \int_{0}^t \partial_t f(u,X_{u})du
%+ \int_{0}^t \nabla_{x} f(u,X_{u-})^\top dX_u \\
%&& \quad +\frac{1}{2} \sum_{i=1}^d \sum_{j=1}^d \int_{0}^t \partial_{x_i}\partial_{x_j} f(u,X_{u-})d[X^i,X^j]_u^c \\
%&& \qquad + \sum_{u\le t} \{f(u,X_u)-f(u,X_{u-})- \nabla_x
%f(u,X_{u-})^\top \Delta X_u\}
%\end{eqnarray*}
%\end{thm}

%\begin{thm}[Feynman-Kac formula]\label{thm:FK} Suppose $X$
%is a Marvov process with infinitesimal generator $L$ and $V$ is
%continuous and bounded. Then
%\[
%u(t,x)=E_{x}\left[ f(X_{t})e^{\int_{0}^{t}V(X_{u})du}\right]
%\]
%satisfies
%\begin{equation}
%\partial_{t}u(t,x)=Lu(t,x)+V(x)u(t,x),\quad u(0,x)=f(x).\label{eq:F-K}
%\end{equation}
%\end{thm}

\begin{thm}[Feynman-Kac formula]\label{thm:FK} Let $\tau$ be a continuous Markov chain on $\{1,\ldots,l\}$, with infinitesimal generator $\Lambda$, and suppose that $f, V\in \R^l$. Then
\begin{equation}\label{eq:F-K}
u_t(i)  = E_{i}\left[ f(\tau_{t}) e^{\int_{0}^{t}V(\tau_{u})du}\right], \quad i\in \{1,\ldots,l\},
\end{equation}
is the unique solution of
\begin{equation}\label{eq:odeFK}
\dot u_t(i) = \partial_{t}u_t(i)= (\Lambda u_t)(i) +V(i)u_t(i), \quad u_0(i)=f(i).
\end{equation}
\end{thm}
\begin{proof}
First, $u$ exists and is unique since \eqref{eq:odeFK} is a system of finite linear ode's. In fact, $u_t = e^{t\{\Lambda+D(v)\}} f$. Next, let $u$ be defined by \eqref{eq:F-K}.
Then, $u$ is bounded, $u_0 = f$, and for any $h\ge0$, since $\tau$ is Markov, $u_{t+h}(i)=E_{i}\left[ u_t(\tau_{h})e^{\int_{0}^{h}V(\tau_{s})ds}\right]$, $i\in \{1,\ldots,l\}$.
Next, if $\zeta = \inf\{s>0;\tau_s \neq i\}$, then $P_i(\zeta>s) = e^{s\Lambda_{ii}}$ and the event that there are at least two state changes in the time interval $[0,h]$ is $o(h)$. As a result,
\begin{eqnarray*}
u_{t+h}(i) &=& \left(1+\Lambda_{ii}h+o(h)\right)u_t(i)e^{hV(i)}+\sum_{j\neq i}\Lambda_{ij} u_t(j)\int_0^h  e^{s\Lambda_{ii}} e^{ sV(i)+(h-s)V(j)}ds+o(h)\\
&=& u_t(i) + h V(i)u_t(i) +h (\Lambda  u_t)_i +o(h).
\end{eqnarray*}
Hence, $u$ satisfies \eqref{eq:odeFK}.
\end{proof}

\subsection{Proof of Lemma \ref{lem:gamma}}
\label{app:lemgamma}
Let $i\in\{1,\ldots,l\}$ be given. We only prove the statements for $\gamma$.  First,
 by the uniqueness of \eqref{eq:odegamma} and the Feynman-Kac formula
in Theorem \ref{thm:FK}, one obtains  that for any $i\in\{1,\ldots,l\}$,
$\gamma_{i}(t)=E_{i}\left\{
e^{-\int_{0}^{t}\ell_{\tau_{u}}du}\right\} $, so $\gamma_{i}(t)\le1$
and $\gamma_{i}(t)>0$. Next,  $\frac{d}{dt}\left\{
\gamma_{i}(t)e^{(\ell_{i}-\Lambda_{ii})t}\right\}
=e^{(\ell_{i}-\Lambda_{ii})t}\sum_{j\neq
i}\Lambda_{ij}\gamma_{j}(t)$. As a result, for all $t\ge0$
$\frac{d}{dt}\left\{
\gamma_{i}(t)e^{(\ell_{i}-\Lambda_{ii})t}\right\} \ge0$, so
$\gamma_{i}(t)\ge e^{(\Lambda_{ii}-\ell_{i})t}$.
\qed

\begin{lem}\label{lem:crochet} Let $(x_{t})$ be a Markov process
in $\R^{l}$ with infinitesimal generator $\L$ and set $\mathcal{L}_tf_t  = \partial_t f_t+ \L f_t$. Suppose
that $M_{t}^{(f)}=f_t(x_{t})-f_0(x_{0})-\int_{0}^{t}\L_u f_u(x_{u})du$,
$M_{t}^{(g)}=g_t(x_{t})-g_0(x_{0})-\int_{0}^{t}\L_u g_u(x_{u})du$, $M_{t}^{(fg)}=f_t(x_t)g_t(x_{t})-f_0(x_0)g_0(x_{0})-\int_{0}^{t}\L_u (f_ug_u)(x_{u})du$
 are martingales.
Then, $
M_{t}^{(f,g)}=M_{t}^{(f_ug_u)}-\int_{0}^{t}g_u(x_{u})dM_{u}^{(f)}-\int_{0}^{t}f_u(x_{u})dM_{u}^{(g)}$
is a martingale and $
M_{t}^{(f,g)}=\left[M^{(f)},M^{(g)}\right]_{t}-\int_{0}^{t}\{\L_u(f_ug_u)-f_u\L_u
g_u-g_u\L_u f_u\}(x_{u})du$.
\end{lem}

\begin{proof} By definition,
$
M_{t}^{(f)}M_{t}^{(g)}=\int_{0}^{t}M_{u}^{(f)}dM_{u}^{(g)}+\int_{0}^{t}M_{u}^{(g)}dM_{u}^{(f)}+\left[M^{(f)},M^{(g)}\right]_{t}$. Setting $V^{(f)}_{t}=\int_{0}^{t}\L_u f_u(x_{u})du$ and
$V^{(g)}_{t}=\int_{0}^{t}\L_u g_u(x_{u})du$, one gets $
M_{t}^{(f)}V^{(g)}_{t}=\int_{0}^{t}\L_u
g_u(x_{u})M_{u}^{(f)}du+\int_{0}^{t}V^{(g)}_{u}dM_{u}^{(f)}$, and $
M_{t}^{(g)}V^{(f)}_{t}=\int_{0}^{t}\L_u
f_u(x_{u})M_{u}^{(g)}du+\int_{0}^{t}V^{(f)}_{u}dM_{u}^{(g)}$. Next, one obtains
%\begin{eqnarray*}
%M_{t}^{(f)}M_{t}^{(g)} & = & f(x_{t})g(x_{t})-g(x_{0})f(x_{t})-f(x_{t})Vg_{t}-f(x_{0})g(x_{t})+f(x_{0})g(x_{0})\\
% &  & \quad+f(x_{0})Vg_{t}-g(x_{t})Vf_{t}+g(x_{0})Vf_{t}+Vf_{t}Vg_{t}\\
% & = & M_{t}^{(fg)}+\int_{0}^{t}\L(fg)(x_{u})du-M_{t}^{(f)}Vg_{t}-M_{t}^{(g)}Vf_{t}\\
% &  & \quad-g(x_{0})M_{t}^{(f)}-g(x_{0})Vf_{t}-f(x_{0})M_{t}^{(g)}-f(x_{0})Vg_{t}-Vf_{t}Vg_{t}-f(x_{0})M_{t}^{g}\\
% & = & M_{t}^{(fg)}+\int_{0}^{t}\L(fg)(x_{u})du-\int_{0}^{t}\{g(x_{u})-M_{u}^{(g)}\}dM_{u}^{(f)}\\
% &  & \quad-\int_{0}^{t}\{f(x_{u})-Vf_{u}\}\L(g)(x_{u})du-\int_{0}^{t}\{g(x_{u})-Vg_{u}\}\L(f)(x_{u})du\\
% &  & \quad-\int_{0}^{t}\{f(x_{u})-M_{u}^{(f)}\}dM_{u}^{(g)}-Vf_{t}Vg_{t}\\
% & = & M_{t}^{(fg)}+\int_{0}^{t}\L(fg)(x_{u})du-\int_{0}^{t}\{g(x_{u})-M_{u}^{(g)}\}dM_{u}^{(f)}\\
% &  & \quad-\int_{0}^{t}f(x_{u})\L(g)(x_{u})du-\int_{0}^{t}g(x_{u})\L(f)(x_{u})du\\
% &  & \quad-\int_{0}^{t}\{f(x_{u})-M_{u}^{(f)}\}dM_{u}^{(g)}\\
% & = & M_{t}^{(fg)}-\int_{0}^{t}g(x_{u})dM_{u}^{(f)}-\int_{0}^{t}f(x_{u})dM_{u}^{(g)}\\
% &  & \quad+\int_{0}^{t}M_{u}^{(f)}dM_{u}^{(g)}+\int_{0}^{t}M_{u}^{(g)}dM_{u}^{(f)}+\int_{0}^{t}\{\L(fg)-f\L g-g\L f\}(x_{u})du\\
% & = & M_{t}^{(fg)}-\int_{0}^{t}g(x_{u})dM_{u}^{(f)}-\int_{0}^{t}f(x_{u})dM_{u}^{(g)}\\
% &  & \quad+M_{t}^{(f)}M_{t}^{(g)}-[M^{(f)},M^{(g)}]_{t}+\int_{0}^{t}\{\L(fg)-f\L g-g\L f\}(x_{u})du.
%\end{eqnarray*}
\begin{eqnarray*}
M_{t}^{(f,g)}  & = & M_{t}^{(fg)}-\int_{0}^{t}g_u(x_{u})dM_{u}^{(f)}-\int_{0}^{t}f_u(x_{u})dM_{u}^{(g)}\\
&=& f_t(x_t)g_t(x_t) - f_0(x_0)g_0(x_0) -\int_0^t \L_u(f_u g_u)(x_u)du \\
&& \qquad -\int_{0}^{t}g_u(x_{u})dM_{u}^{(f)}-\int_{0}^{t}f_u(x_{u})dM_{u}^{(g)}\\
&=& \left(M_t^{(f)}+V_t^{(f)}\right)  \left(M_t^{(g)}+V_t^{(g)}\right) +  f_0(x_0)V_t^{(g)} +g_0(x_0)V_t^{(f)} \\
&& \qquad -\int_{0}^{t}\{g_u(x_{u})-g_0(x_0)\}dM_{u}^{(f)}-\int_{0}^{t}\{f_u(x_{u})-f_0(x_0)\}dM_{u}^{(g)}\\
&& \qquad \qquad -\int_0^t \L_u(f_u g_u)(x_u)du\\
&=& \left(M_t^{(f)}+V_t^{(f)}\right)  \left(M_t^{(g)}+V_t^{(g)}\right) +  f_0(x_0)V_t^{(g)} +g_0(x_0)V_t^{(f)} \\
&& \qquad -\int_{0}^{t} \left(M_u^{(g)}+V_u^{(g)}\right)dM_{u}^{(f)}-\int_{0}^{t} \left(M_u^{(f)}+V_u^{(f)}\right)dM_{u}^{(g)}\\
&& \qquad \qquad -\int_0^t \L_u(f_u g_u)(x_u)du\\
 & = & [M^{(f)},M^{(g)}]_{t} - \int_{0}^{t}\{\L_u(f_ug_u)-f_u\L_u g_u-g_u\L_u f_u\}(x_{u})du.
\end{eqnarray*}
Hence the result.
\end{proof}
Note that if $g(s) = s$, then for any smooth function $f$,
\begin{equation}
\left\{ \L_{i}(fg)-f\L_{i}g-g\L_{i}f\right\} (s,i)=D(s)a(i)D(s)\nabla f(s,i).\label{eq:crochetg}
\end{equation}

\section{Proof of the main results}\label{app:main}
%%%%%%%%%%%%%%%%%%%%%%%%%%%%%%%%%%%%%%%%%%%%%%%%%%%%%%%%%%%%%%%%%%%%%%%%%%%%%%%%
%%%%%%%%%%%%%%%%%%%%%%%%%%%%%%%%%%%%%%%%%%%%%%%%%%%%%%%%%%%%%%%%%%%%%%%%%%%%%%%%

%%%%%%%%%%%%%%%%%%%%%%%%%%%%%%%%%%%%%%%%%%%%%%%%%%%%%%%%%%%%%%%%%%%%%%%%%%%%%%%%
\subsection{Proof of Lemma \ref{lem:changeofmeasure0}}\label{app:changeofmeasure0}
%%%%%%%%%%%%%%%%%%%%%%%%%%%%%%%%%%%%%%%%%%%%%%%%%%%%%%%%%%%%%%%%%%%%%%%%%%%%%%%%
The multiplicative character of
$Z$ follows directly from the representation $Z_t = e^{-M_t - \frac{1}{2}\int_0^t \ell_{\tau_u}du}$,
$
M_{t}=\int_{0}^{t}\ell_{\tau_{u}}du+\int_{0}^{t}\rho(\tau_{u-})^{\top}\sigma(\tau_{u-})dW_{u}$,
since  $\tau$ and $(W,\tau)$ are Markov processes. Hence, for any $s,t\ge 0$,
$Z_{t+s} = Z_t Z_{t,s}$ where for a fixed $t$, $Z_{t,s}$ is independent of $\calf_t$ and has the same law as
$Z_s$. Next, from Lemma \ref{lem:gamma},
$Z^{(2)}$ is a positive martingale. Next, since $W$ and $\tau$ are independent and $\calf_t = \calf_t^W \vee \calf_t^\tau \subset \calf_t^W \vee \calf_T^\tau$, it follows that for any $t\in [0,T]$,
\begin{eqnarray*}
E[Z_T|\calf_t] &=& E\left[Z_T^{(1)}  Z_T^{(2)}|\calf_t\right]  = E\left[Z_T^{(2)} E\left[Z_T^{(1)} |\calf_t^W \vee \calf_T^\tau \right]|\calf_t\right]\\
&=& E\left[Z_t^{(1)} Z_T^{(2)}|\calf_t\right] = Z_t^{(1)}  E\left[ Z_T^{(2)}|\calf_t\right] = Z_t^{(1)}Z_t^{(2)}=Z_{t}\gamma_{\tau_{t}}(T-t),\\
\end{eqnarray*}
proving that $Z_t^{(1)}Z_t^{(2)}$ is a martingale for $t\in [0,T]$. Similarly,  from Lemma \ref{lem:gamma},
$$
E_{i}(Z_{t}) = E_i\left[ e^{-\int_{0}^{t}\ell_{\tau_{u}}du}E\left[ Z_t^{(1)} |\calf_t^\tau\right]\right]=E_i\left[ e^{-\int_{0}^{t}\ell_{\tau_{u}}du}\right]=
\gamma_{i}(t), \quad t\ge 0.
$$
Now take $\theta\in \R^d$, $f\in \R^l$, and let $0\le s\le t \le T$ be given. Then
\begin{eqnarray*}
E^{\tilde P}\left[ f(\tau_t) e^{\theta^\top (\tilde W_t-\tilde W_s)}|\calf_s\right] &=& E\left[ Z_t^{(1)} Z_t^{(2)} f(\tau_t) e^{\theta^\top (\tilde W_t -\tilde W_s)} |\calf_s\right]/(Z_s^{(1)} Z_s^{(2)})\\
&=& E\left[  Z_t^{(2)} f(\tau_t) E\left[ Z_t^{(1)}e^{\theta^\top (\tilde W_t-\tilde W_s)} |\calf_s^W \vee \calf_t^\tau\right]|\calf_s\right]/(Z_s^{(1)} Z_s^{(2)}).
\end{eqnarray*}
Now, set $h_u = \sigma_{\tau_u}^\top \rho_{\tau_u}$.  Novikov's condition is clearly satisfied for $h$ and one can check that
$$
E\left[ Z_t^{(1)}e^{\theta^\top \tilde W_t} |\calf_s^W \vee \calf_t^\tau\right] = e^{\|\theta\|^2 (t-s)} Z_s^{(1)}e^{\theta^\top \tilde W_s}.
$$
Hence
\begin{eqnarray*}
E^{\tilde P} \left[ f(\tau_t) e^{\theta^\top (\tilde W_t-\tilde W_s)} |\calf_s\right] &= & e^{\|\theta\|^2 (t-s)} E\left[  Z_t^{(2)} f(\tau_t) \right]/Z_s^{(2)}\\
&= & e^{\|\theta\|^2 (t-s)} E\left[ \gamma_{\tau_t}(T-t)   e^{-\int_s^t \ell_{\tau_u} du} f(\tau_t) \right]/\gamma_{\tau_s}(T-s)\\
&=& e^{\|\theta\|^2 (t-s)} E_{\tau_s} \left[ \gamma_{\tau_{t-s}}(T-t)   e^{-\int_0^{t-s} \ell_{\tau_u} du} f(\tau_{t-s}) \right]/\gamma_{\tau_s}(T-s).
\end{eqnarray*}
%Finally, note that for any smooth $f$, one gets
%\begin{eqnarray*}
%E^{\tilde P}\left[f(S_{t+h},\tau_{t+h})|\calf_t  \right] & =&  \frac{E\left[Z_{t+h}\gamma_{\tau_{t+h}}(T-t-h)f(S_{t+h},\tau_{t+h})|\calf_t  \right]}{Z_t\gamma_{\tau_t}(T-t)}\\
%&=& \frac{E_{(S_t,\tau_t)}\left[Z_{h}\gamma_{\tau_{h}}(T-t-h)f(S_{h},\tau_{h})\right]}{\gamma_{\tau_t}(T-t)}.
%\end{eqnarray*}
%Using Lemma \ref{lem:zshmm} with $\tilde f_u(s,i)=\gamma_{i}(T-t-u) f(s,i)$, one
%gets
%$$
%\partial_u \tilde f_u (s,i)+ {\tilde {\mathcal{H }}} {\tilde f}_u(s,i)= -{\dot \gamma}_i(T-u) f(s,i) + \mathcal{H} \{f\tau (T-u)\}(s,i) = \gamma_i(T-u){\tilde {\mathcal{H }}}_u f (s,i),
%$$
%using $\dot{\gamma}_{i}(t)=-\ell_{i}\gamma_{i}(t)+\Lambda_{ii}\gamma_{i}(t)-(\tilde{\Lambda}_{t})_{ii}\gamma_{i}(t)$, from \eqref{eq:lambdat}.
%This completes the proof.
This proves that under $\tilde P$, $\tilde W$  is a Brownian motion independent of $\tau$, and consequently,
\begin{eqnarray*}
X_t &=& X_0 + \int_0^t D(X_u)m_{\tau_u}du +\int_0^t D(X_u)\sigma_{\tau_{u-}}dW_u\\
&=& X_0 + \int_0^t D(X_u)\left( m_{\tau_u}- \sigma_{\tau_u}\sigma_{\tau_u}^\top \rho_{\tau_u}\right)du +\int_0^t D(X_u)\sigma_{\tau_{u-}}d\tilde W_u\\
&=& X_0 +\int_0^t D(X_u)\sigma_{\tau_{u-}}d\tilde W_u,
\end{eqnarray*}
proving that $X$ is a $\tilde P$-martingale. Similarly, $
S_t = S_0 + \int_0^t r_{\tau_u} S_u du  +\int_0^t D(S_u)\sigma_{\tau_{u-}}d\tilde W_u$.
Finally, if $0\le t\le t+h\le T$, then
\begin{eqnarray*}
E^{\tilde P} \left[ f(\tau_{t+h}) |\calf_t\right] &=& E_{\tau_t} \left[ \gamma_{\tau_{h}}(T-t-h)   e^{-\int_0^{h} \ell_{\tau_u} du} f(\tau_{h}) \right]/\gamma_{\tau_t}(T-t)\\
  &=& E_{\tau_t} \left[ \gamma_{\tau_{h}}(T-t)   e^{-\int_0^{h} \ell_{\tau_u} du} f(\tau_{h}) \right]/\gamma_{\tau_t}(T-t)\\
  && \quad
-h E_{\tau_t} \left[ \dot \gamma_{\tau_{h}}(T-t)   e^{-\int_0^{h} \ell_{\tau_u} du} f(\tau_{h}) \right]/\gamma_{\tau_t}(T-t)+o(h).
\end{eqnarray*}

It follows from  Theorem \ref{thm:FK} that if $v_h(i) =  E_{i} \left[ \gamma_{\tau_{h}}(T-t)   e^{-\int_0^{h} \ell_{\tau_u} du} f(\tau_{h}) \right]$, then
\begin{eqnarray*}
v_h(i) &=&  \gamma_i(T-t)f(i) + h \sum_{j=1}^l \Lambda_{ij} \gamma_j(T-t) f(j)- h \ell_{i} \gamma_i(T-t)f(i) +o(h) \\
&=&  \gamma_i(T-t)f(i) + h \gamma_i(T-t) \sum_{j=1}^l \left({\tilde \Lambda}_{T-t}\right)_{ij} f(j) \\
&& \qquad +h \gamma_i(T-t)f(i) \left\{ \Lambda_{ii}-\left({\tilde \Lambda}_{T-t}\right)_{ii}  -\ell_{i}  \right\} +o(h) \\
&=&  \gamma_i(T-t)f(i) + h \gamma_i(T-t) \sum_{j=1}^l \left({\tilde \Lambda}_{T-t}\right)_{ij} f(j) +h \dot \gamma_i(T-t)f(i)+ o(h),
\end{eqnarray*}
using $\dot{\gamma}_{i}(t)=-\ell_{i}\gamma_{i}(t)+\Lambda_{ii}\gamma_{i}(t)-(\tilde{\Lambda}_{t})_{ii}\gamma_{i}(t)$, from \eqref{eq:lambdat}. Therefore, if $\tau_t=i$, then
$$
E^{\tilde P} \left[ f(\tau_{t+h}) |\calf_t\right]= f(i) +h  {\tilde \Lambda}_{T-t} f (i)+o(h),
$$
showing that under $\tilde P$, $\tau$ is a Markov chain with generator ${\tilde \Lambda}_{T-t}$. \qed

%%%%%%%%%%%%%%%%%%%%%%%%%%%%%%%%%%%%%%%%%%%%%%%%%%%%%%%%%%%%%%%%%%%%%%%%%%%%%%%%
\subsection{Proof of Lemma \ref{lem:changeofmeasure}}\label{app:changeofmeasure}
%%%%%%%%%%%%%%%%%%%%%%%%%%%%%%%%%%%%%%%%%%%%%%%%%%%%%%%%%%%%%%%%%%%%%%%%%%%%%%%%
%Using Lemma \ref{lem:zshmm} with $f_t(s,i)=\gamma_{i}(T-t)s$, one
%gets $\partial_{t}f_{t}(s,i)+\tilde{\mathcal{H}}f_t(s,i)=r_i f_t(s,i)$,
%so $Y_{t}-Y_{0}-\int_{0}^{t}r_{\tau_{u}}Y_{u}du$ and $B_{t}Y_{t}=X_{t}\gamma_{\tau_{t}}(T-t)Z_{t}$ are martingales. Next,
%$
%From Lemma \ref{lem:changeofmeasure0}, $
%\frac{E\{X_{T}Z_{T}|\calf_{t}\}}{E(Z_{T}|\calf_{t})}=\frac{E\left\{ B_{T}Y_{T}|\calf_{t}\right\} }{E(Z_{T}|\calf_{t})}=\frac{B_{t}Y_{t}}{Z_t\gamma_{\tau_{t}}(T-t)}=X_{t}$.
%Using \eqref{eq:Xmart}, one can write
%\[
%M_{t}=\int_{0}^{t}\ell_{\tau_{u}}du+\int_{0}^{t}\rho(\tau_{u-})^{\top}\sigma(\tau_{u-})dW_{u}=\int_{0}^{t}\ell_{\tau_{u}}du+\mathcal{M}_{t},
%\]
%where $\mathcal{M}$ is a martingale. As a result, $Z_{t}=Z_{t}^{(1)}e^{-\int_{0}^{t}\ell_{\tau_{u}}du}$,
%where $Z^{(1)}=\mathcal{E}(-\mathcal{M})$ is a martingale. Also, from Lemma \ref{lem:gamma},
%\begin{equation}\label{eq:Z2}
%Z^{(2)}(t)=E\left\{\left. e^{-\int_{0}^{T}\ell_{\tau_{u}}du}\right|\calf_{t}\right\} =\gamma_{\tau_{t}}(T-t)e^{-\int_{0}^{t}\ell_{\tau_{u}}du},\quad t\in[0,T],
%\end{equation}
%is a positive martingale with initial value $\gamma_{\tau_{0}}(T)$.
%Since $Z_{T}=Z_{T}^{(1)}Z_{T}^{(2)}$,
%\eqref{eq:Zgamma} yields
%\[
%E\left(Z_{T}^{(1)}Z_{T}^{(2)}|\calf_{t}\right)=E(Z_{T}|\calf_{t})=Z_{t}\gamma_{\tau_{t}}(T-t)=Z_{t}^{(1)}\gamma_{\tau_{t}}(T-t)
%e^{-\int_{0}^{t}\ell_{\tau_{u}}du}=Z_{t}^{(1)}Z_{t}^{(2)},
%\]
%proving that $\left[Z^{(1)},Z^{(2)}\right]=0$. The last result
%follows from Remark \ref{rem:Zgamma}.
%Finally, to
To prove that $\tilde P$ is variance-optimal, set
$b_t = D^{-1}(X_t)\rho(\tau_{t-})$, so that $M_t = \int_0^t b_u^\top dX_u$. It is then easy to check that $\bar b = b Z$ is such that $\int_0^T \bar b_u ^\top dX_u$ is square integrable, and since $X$ is a $\tilde P$-martingale, then for any bounded stopping times $U,V$ with $0\le U\le V\le T$, and any bounded $\calf_U$-measurable variable $H$, one gets $E\left[Z_T H(X_V-X_U)\right] =0$, proving that $b$ is an adjustment process, as defined in  \citet{Schweizer:1996}. It follows from \citet[Proposition 8]{Schweizer:1996} that $\tilde P$ is variance-optimal.
\qed

\subsection{Proof of Lemma \ref{lem:checkP}}\label{app:checkP}
If $f(s,i)$ is twice continuously differentiable with respect
to $s$ and $(S_{t},{\tau}_{t})=(s,i)$,  then
\begin{eqnarray*}
E^{\check{P}}\left[ f(S_{t+h},\tau_{t+h})|\calf_{t}\right]  & = & \frac{\beta_{i}(T-t)}{s}E^{\inw P}\left[ S_{T}f(S_{t+h},\tau_{t+h})|\calf_{t}\right] \\
 & = & \frac{\beta_{i}(T-t)}{s}\frac{E^{\tilde{P}}\left[ B_{T}S_{T}f(S_{t+h},\tau_{t+h})|\calf_{t}\right] }{E^{\tilde{P}}\left[ B_{T}|\calf_{t}\right] }\\
 & = & \frac{1}{s B_t}E^{\tilde{P}}\left[ X_{T}f(S_{t+h},\tau_{t+h})|\calf_{t}\right] = \frac{1}{sB_{t}}E^{\tilde{P}}\left[ X_{t+h}f(S_{t+h},\tau_{t+h})|\calf_{t}\right] .
\end{eqnarray*}
Then, with $g$ being the identity function, if  $(S_{t},{\tau}_{t})=(s,i)$, it follows that
\begin{eqnarray*}
\lim_{h\downarrow0}\frac{1}{h}\left[E^{\check{P}}\left[ f(S_{t+h},\tau_{t+h})|\calf_{t}\right] -f(s,i)\right] & = & -r_{i}f(s,i)+\frac{\tilde{\cH}_{t}(fg)(s,i)}{s} =\check{\mathcal{H}}_{t}f(s,i).
\end{eqnarray*}
\qed

\subsection{Proof of Lemma \ref{lem:alpha}}

\label{app:alpha} %%%%%%%%%%%%%%%%%%%%%%%%%%%%%%%%%%%%%%%%%%%%%%%%%%%%%%%%%%%%%%%%%%%%%%%%%%%%%%%%

According to \citet{Protter:2004}{[}Theorem V.7{]}, the solution
 of
$
V_{t}=C_0(s,i)+\int_{0}^{t}\alpha_u(S_{u},\tau_{u-})^{\top}dX_{u}-\int_{0}^{t}V_{u-}dM_{u}
$
exists and is uniquely determined by $M$ and $S$. See also Remark
\ref{rem:Vrep} below.
Next,  $\phi_{t}=\alpha_t(S_{t-},\tau_{t-})-V_{t-}D^{-1}(X_{t-})\rho(\tau_{t-})$, so
$\phi$ is predictable and $V_{t}=C_0(s,i)+\int_{0}^{t}\phi_{u}^{\top}dX_{u}$,
by definition of $M$.  \qed

\begin{RM}\label{rem:Vrep} Since the martingale $M$ is continuous,
\citet{Protter:2004}{[}Theorem V.52{]} yields the
representation  $ V_{t}=Z_{t}\left\{
H_{0}+\int_{0+}^{t}Z_{u}^{-1}d(H_{u}+[H,M]_{u})\right\}$,  where
\begin{eqnarray*}
H_{t} & = & C_{0}(s,i)+\int_{0}^{t}\alpha_u(S_{u},\tau_{u-})^{\top}dX_{u}\\
 & = & C_{0}(s,i)+\int_{0}^{t}\alpha_u(S_{u},\tau_{u-})^{\top}D(X_{u})m(\tau_{u})du\\
 &  & \quad+\int_{0}^{t}\alpha_u(S_{u},\tau_{u-})^{\top}D(X_{u})\sigma(\tau_{u-})dW_{u},
\end{eqnarray*}
and $[H,M]_{t}=\int_{0}^{t}\alpha_u(S_{u},\tau_{u})^{\top}D(X_{u})m(\tau_{u})du$.
As a result,
\begin{eqnarray}
\frac{V_{t}}{Z_t} & = &
C_{0}+2\int_{0}^{t}\frac{\alpha_u(S_{u},\tau_{u})^{\top}D(X_{u})m(\tau_{u})}{Z_{u}}du\label{eq:Vrep2}
+\int_{0}^{t}\frac{\alpha_u(S_{u},\tau_{u})^{\top}D(X_{u})\sigma(\tau_{u})}{Z_{u}}dW_{u}.\nonumber
\end{eqnarray}

\end{RM}

%%%%%%%%%%%%%%%%%%%%%%%%%%%%%%%%%%%%%%%%%%%%%%%%%%%%%%%%%%%%%%%%%%%%%%%%%%%%%%%%

\subsection{Proof of Lemma \ref{lem:Grep}}

\label{app:pfGrep} %%%%%%%%%%%%%%%%%%%%%%%%%%%%%%%%%%%%%%%%%%%%%%%%%%%%%%%%%%%%%%%%%%%%%%%%%%%%%%%%
Let $M^{(C)}$ and $M^{(g)}$ be the martingales  defined
by $ M_{t}^{(C)}=
C_{t}(S_{t},\tau_{t})-C_{0}(s,i)-\int_{0}^{t}\left\{\partial_u C_u(S_{u},\tau_{u})+\mathcal{H}C_{u}(S_{u},\tau_{u})\right\}du$
and $ M_{t}^{(g)} =
S_{t}-s-\int_{0}^{t}D(S_{u})\mu(\tau_{u})du=\int_{0}^{t}D(S_{u})\sigma(\tau_{u-})dW_{u}$, with $g(s) = s$ on $[0,\infty)^d$.
As a result, $
X_{t}=s+\int_{0}^{t}D(X_{u})m(\tau_{u})du+\int_{0}^{t}B_{u}dM_{u}^{(g)}
$,
\begin{eqnarray*}
V_{t} & = & C_0(s,i)+\int_{0}^{t}B_{u}\phi_{u}(S_{u-},\tau_{u-})^{\top}dM_{u}^{(g)}\\
 &  & \quad+\int_{0}^{t}B_{u}m(\tau_{u})^{\top}D(S_{u})\phi_{u}(S_{u},\tau_{u})du\\
 & = & C_0(s,i)+\int_{0}^{t}B_{u}\phi_{u}(S_{u-},\tau_{u-})^{\top}dM_{u}^{(g)}\\
 &  & \quad+\int_{0}^{t}B_{u}m(\tau_{u})^{\top}D(S_{u})\nabla
 C_{u}(S_{u},\tau_{u})du +\int_{0}^{t}G_{u}\ell_{\tau_{u}}du
\end{eqnarray*}
\begin{eqnarray*}
B_{t}C_t(S_{t},\tau_{t}) & = & C_0(s,i)+\int_{0}^{t}B_{u}\left\{ \partial_{u}C_(S_{u},\tau_u)-rC_u(S_{u},\tau_u)\right\} du\\
 &  & \quad+\int_{0}^{t}B_{u}\cH C_u(S_{u},\tau_{u})du+\int_{0}^{t}B_{u}dM_{u}^{(C)}.
\end{eqnarray*}
Therefore,  using \eqref{eq:odegammanew} and using \eqref{eq:phidef2} one obtains that
\begin{eqnarray}
G_{t} & = & -\int_{0}^{t}B_{u}{\tilde{\mathcal{H}}}_{T-u}C_{u}(S_{u},\tau_{u})du+\int_{0}^{t}B_{u}\mathcal{H}C_{u}(S_{u},\tau_{u})du+\int_{0}^{t}B_{u}dM_{u}^{(C)} \nonumber\\
 &  & \quad-\int_{0}^{t}B_{u}m(\tau_{u})^{\top}D(S_{u})\nabla C_{u}(S_{u},\tau_{u})du-\int_{0}^{t}B_{u}\phi_{u}(S_{u},\tau_{u})^{\top}dM_{u}^{(g)}\nonumber\\
 &  & \qquad-\int_{0}^{t}\ell_{\tau_{u}}G_udu\nonumber\\
 & = & \int_{0}^{t}B_{u}dM_{u}^{(C)}-\int_{0}^{t}B_{u}\phi_{u}(S_{u},\tau_{u})^{\top}dM_{u}^{(g)}-\int_{0}^{t}\ell_{\tau_{u}}G_{u}du\nonumber\\
 &  & \quad +\int_{0}^{t}B_{u}\left\{ (\Lambda-\tilde{\Lambda}_{T-u})C_{u,S_{u}}\right\} (\tau_{u})du.\label{eq:G2}
\end{eqnarray}
%
%By Lemma \ref{lem:crochet} and Equation \eqref{eq:crochetg}, there
%exists a martingale $M^{(C,g)}$ so that
%\begin{eqnarray*}
%\left[M^{(C)},M^{(g)}\right]_{t} & = & M_{t}^{(C,g)}+\int_{0}^{t}\left\{ \cH(gC_{t})-g\cH C_{t}-C_{t}\cH g\right\} (S_{u},\tau_{u})du\\
% & = & M_{t}^{(C,g)}+\int_{0}^{t}D(S_{u})a(\tau_{u})D(S_{u})\nabla C_{u}(S_{u},\tau_{u})du.
%\end{eqnarray*}
Finally, it follows from Ito's formula  that
\begin{eqnarray*}
M_{t}^{(C)} & = & \int_{0}^{t}\nabla_{s}C_u(S_{u},\tau_{u})^{\top}D(S_{u})\sigma(\tau_{u-})dW_{u}\\
 &  & \qquad+\sum_{0<u\le t}\Delta C_u(S_{u},\tau_{u})-\int_{0}^{t}\Lambda C_u(S_{u},\tau_{u})du.
\end{eqnarray*}
Replacing in the above formula yields the desired result. \qed

%%%%%%%%%%%%%%%%%%%%%%%%%%%%%%%%%%%%%%%%%%%%%%%%%%%%%%%%%%%%%%%%%%%%%%%%%%%%%%%%
%%%%%%%%%%%%%%%%%%%%%%%%%%%%%%%%%%%%%%%%%%%%%%%%%%%%%%%%%%%%%%%%%%%%%%%%%%%%%%%%
\subsection{Proof of Lemma \ref{lem:GSmart}}
\label{app:pfGSmart}
%%%%%%%%%%%%%%%%%%%%%%%%%%%%%%%%%%%%%%%%%%%%%%%%%%%%%%%%%%%%%%%%%%%%%%%%%%%%%%%%
%%%%%%%%%%%%%%%%%%%%%%%%%%%%%%%%%%%%%%%%%%%%%%%%%%%%%%%%%%%%%%%%%%%%%%%%%%%%%%%%
Let $g(s)=s$  and let $M^{(g)}$ be as defined in the proof of Lemma \ref{lem:Grep}.  Then, using \eqref{eq:G2} and It\^{o}'s formula, one gets, for any smooth $f$,
\begin{eqnarray*}
G_{t}f_{t}(S_{t},\tau_{t}) & = & \int_{0}^{t}f_{u}(S_{u-},\tau_{u-})dG_{u}+\int_{0}^{t}\partial_{u}f_{u}(S_{u},\tau_{u})G_{u}du\\
 &  & \quad+\int_{0}^{t}\mathcal{H}f_{u}(S_{u},\tau_{u})G_{u}du+\int_{0}^{t}G_{u-}dM_{u}^{(f)}+\left[G,M^{(f)}\right]_{t}\\
 & = & \int_{0}^{t}B_{u}f_{u}(S_{u-},\tau_{u-})dM_{u}^{(C)}+\int_{0}^{t}G_{u-}dM_{u}^{(f)}+\left[G,M^{(f)}\right]_{t}\\
 &  & \quad-\int_{0}^{t}B_{u}f_{u}(S_{u-},\tau_{u-})\phi_{u}^{\top}dM_{u}^{(g)}\\
 &  & \qquad+\int_{0}^{t}B_{u}f_{u}(S_{u},\tau_{u})\left\{ (\Lambda-\tilde{\Lambda}_{T-u})C_{u,S_{u}}\right\} (\tau_{u})du\\
 &  & \qquad+\int_{0}^{t}\left\{ \mathcal{H}f_{u}(S_{u},\tau_{u})+\partial_{u}f_{u}(S_{u},\tau_{u})-f_{u}(S_{u},\tau_{u})\ell_{\tau_{u}}\right\} G_{u}du,
\end{eqnarray*}
where
$
M_{t}^{(f)}=f_{t}(S_{t},\tau_{t})-f_{0}(s,i)-\int_{0}^{t}\partial_{u}f_{u}(S_{u},\tau_{u})du-\int_{0}^{t}\mathcal{H}f_{u}(S_{u},\tau_{u})du
$
is a martingale and where, by  \eqref{eq:G2} and Lemma \ref{lem:crochet}, one has
\begin{eqnarray*}
\left[G,M^{(f)}\right]_{t} & = &
\int_{0}^{t}B_{u}d\left[M^{(C)},M^{(f)}\right]_{u} -\int_{0}^{t}B_{u}\phi_{u}^{\top}d\left[M^{(g)},M^{(f)}\right]_{u}\\
 & = & \int_{0}^{t}B_{u}dM_{u}^{(C,f)}-\int_{0}^{t}B_{u}\phi_{u}^{\top}dM_{u}^{(f,g)}\\
 &  & \quad+\int_{0}^{t}B_{u}\{\mathcal{H}(f_{u}C_{u})-f_{u}\mathcal{H}(C_{u})-C_{u}\mathcal{H}(f_{u})\}(S_{u},\tau_{u})du\\
 &  & \qquad-\int_{0}^{t}B_{u}\phi_{u}^{\top}\{\mathcal{H}(f_{u}g)-f_{u}\mathcal{H}(g)-g\mathcal{H}(f_{u})\}(S_{u},\tau_{u})du.
\end{eqnarray*}
As a result, if $f_{t}(s,i)=\gamma_{T-t}(i)$, then
$\mathcal{H}f_{u}(s,i)+\partial_{u}f_{u}(s,i)-f_{u}(s,i)\ell_{i}\equiv0$,
and $
\mathcal{N}(f_{u},g)(s,i)=\{\mathcal{H}(f_{u}g)-f_{u}\mathcal{H}(g)-g\mathcal{H}(f_{u})\}(s,i)\equiv0
$,
using \eqref{eq:crochetg}.
Moreover
\begin{eqnarray*}
\mathcal{N}(f_{u},C_{u})(s,i) & = & \left\{ \mathcal{H}(f_{u}C_{u})-f_{u}\mathcal{H}(C_{u})-C_{u}\mathcal{H}(f_{u})\right\} (s,i)\\
 & = & -\gamma_{T-u}(i)\Lambda C_{u,s}(i)-C_{u,s}(i)\Lambda\gamma_{T-u}(i)+\Lambda(\gamma_{T-u}C_{u,s})(i)\\
 & = & -f_{u}(s,i)\left\{ (\Lambda-\tilde{\Lambda}_{T-u})C_{u,s}\right\} (i),
\end{eqnarray*}
since $\Lambda(\gamma_{t}h)(i)-\gamma_{t}(i)\Lambda
h(i)-h(i)\Lambda\gamma_{t}(i)=\gamma_{t}(i)(\tilde{\Lambda}_{t}-\Lambda)h(i)$,
$i\in\{1,\ldots,l\}$. Hence $G_{t}\gamma_{\tau_{t}}(T-t)$ is a
martingale with initial value $0$ and terminal value $G_{T}$, since
$\gamma(0)=\mathbf{1}$. Thus, $E(G_{T})=0$. Next, take
$f_{t}(s,i)=s_{k}\gamma_{T-t}(i)$, for a given $k\in\{1,\ldots,d\}$.
Then
$$
\mathcal{N}(f_{u},g)(s,i) = \{\mathcal{H}(f_{u}g)-f_{u}\mathcal{H}(g)-g\mathcal{H}(f_{u})\}(s,i)=f_{u}(s,i)D(s)a(i)e_{k}
$$
by \eqref{eq:crochetg}, with $(e_{k})_{j}=I_{jk}$. Also,
$\mathcal{H}f_{u}(s,i)+\partial_{u}f_{u}(s,i)-f_{u}(s,i)\ell_{i}=f_{u}(s,i)\mu_{k}(i)$
and
\begin{eqnarray*}
\mathcal{N}(f_{u},C_{u})(s,i) & = & \left\{ \mathcal{H}(f_{u}C_{u})-f_{u}\mathcal{H}(C_{u})-C_{u}\mathcal{H}(f_{u})\right\} (s,i)\\
 & = & f_{u}(s,i)e_{k}^{\top}a(i)D(s)\nabla C_{u}(s,i)-f_{u}(s,i)\left\{ (\Lambda-\tilde{\Lambda}_{T-u})C_{u,s}\right\} (i),
\end{eqnarray*}
using the previous calculations. It follows that $R_t - A_t$ is a martingale, where
\begin{multline*}
A_t = \int_0^t B_u \phi_u ^\top \left\{\mathcal{N}(f_u,C_u)-\mathcal{N}(f_u,g) \right\}(S_u,\tau_u) du \\
+\int_{0}^{t}B_{u}f_{u}(S_{u},\tau_{u})\left\{ (\Lambda-\tilde{\Lambda}_{T-u})C_{u,S_{u}}\right\} (\tau_{u})du\\
+\int_{0}^{t}\left\{ \mathcal{H}f_{u}(S_{u},\tau_{u})+\partial_{u}f_{u}(S_{u},\tau_{u})-f_{u}(S_{u},\tau_{u})\ell_{\tau_{u}}\right\} G_{u}du.
\end{multline*}
Using \eqref{eq:phidef2} and setting $R_{t}=f_{t}(S_{t},\tau_{t})G_{t}$, one gets that
\begin{eqnarray*}
A_t &=& \int_0^t \left[R_u \{ m_k(\tau_u) + r_{\tau_u}\} + f_u(S_u,\tau_u) e_k^\top a(\tau_u)D(X_u)\{\nabla C_u(S_u,\tau_u)-\phi_u)\}\right]du \\
&=& \int_0^t \left[R_u \{ m_k(\tau_u) + r_{\tau_u}\} - R_u e_k^\top a(\tau_u)\rho(\tau_u)\right]du = \int_0^t r_{\tau_u}R_u du,
\end{eqnarray*}
proving that  $R_{t}-\int_{0}^{t} r_{\tau_u} R_{u}du$ is a martingale. Hence, so is $B_{t}R_{t}=X_{t}G_{t}\gamma_{T-t}(\tau_{t})$.
Finally, to prove \eqref{eq:optcont}, note that $E(G_{T}|\calf_{t})=\gamma_{\tau_{t}}(T-t)G_{t}$. Thus,
for any stopping times $U,V$ with $0\le U\le V\le T$,
\begin{eqnarray*}
E\{G_{T}(X_{V}-X_{U})|\calf_{u}\} & = & E\{(X_{V}E(G_{T}|\calf_{V})|\calf_{U}\}-X_{U}E(G_{T}|\calf_{U})\\
 & = & E\{G_{V}\gamma_{\tau_{V}}(T-V)X_{V}-G_{U}\gamma_{\tau_{U}}(T-U)X_{U}|\calf_{U}\}=0,
\end{eqnarray*}
since we just proved that $X_{t}G_{t}\gamma_{\tau_{t}}(T-t)$ is a
martingale. Hence the result.\qed

%%%%%%%%%%%%%%%%%%%%%%%%%%%%%%%%%%%%%%%%%%%%%%%%%%%%%%%%%%%%%%%%%%%%%%%%%%%%%%%%%
%%%%%%%%%%%%%%%%%%%%%%%%%%%%%%%%%%%%%%%%%%%%%%%%%%%%%%%%%%%%%%%%%%%%%%%%%%%%%%%%%

\subsection{Proof of Theorem \ref{thm:optcont}}

\label{pf:optcont} %%%%%%%%%%%%%%%%%%%%%%%%%%%%%%%%%%%%%%%%%%%%%%%%%%%%%%%%%%%%%%%%%%%%%%%%%%%%%%%%%
%%%%%%%%%%%%%%%%%%%%%%%%%%%%%%%%%%%%%%%%%%%%%%%%%%%%%%%%%%%%%%%%%%%%%%%%%%%%%%%%%
It follows from \eqref{eq:optcont} that for any stopping times
$U,V$, with $0\le U\le V\le T$, and any bounded random variable $Y$
that is $\calf_{U}$-measurable, $
E\left(G_{T}\int_{0}^{T}\psi_{t}^{\top}dX_{t}\right)=0$, where
$\psi$ is the predictable process given by $\psi_{t}=Y1_{(U,V]}(t)$.
Therefore, using properties of stochastic integrals, one may
conclude that for any admissible strategy $\psi$, one gets $
E\left(G_{T}\int_{0}^{T}\psi_{t}^{\top}dX_{t}\right)=0$. Since
$E\left\{ G_{T}\int_{0}^{T}(\phi_{t}-\psi_{t})^{\top}dX_{t}\right\}
=0$ and $E(G_{T})=0$, one has, for any $(\pi,\psi) \in \mathbb{R}\times\mathcal{A}$,
$$
\mathit{HE}(\pi,\psi) =  \mathit{HE}(C_0,\phi) + E\left[\left\{
C_{0}(s,i)-\pi+\int_{0}^{T}(\phi_{t}-\psi_{t})^{\top}dX_{t}\right\}
^{2}\right]\ge \mathit{HE}(C_0,\phi).
$$
This proves that at least one solution exists. Suppose now that $\mathit{HE}(\pi,\psi) = \mathit{HE}(C_0,\phi)$ for some $(\pi,\psi)\in \R\times \A$. Then
$C_{0}(s,i)-\pi+\int_{0}^{T}(\phi_{t}-\psi_{t})^{\top}dX_{t} =0 $ $P$ almost surely. Since $\tilde P$ is equivalent to $P$ and $X$ is a $\tilde P$-martingale, it follows that $N_t =
\int_{0}^{t}(\phi_{s}-\psi_{s})^{\top}dX_{s}$ is a $\tilde P$-martingale. Hence its expectation is $0$, implying that so $C_0=\pi$. It then follows that $N_t \equiv 0$ $P$-a.s. Finally, the $\tilde P$-martingale
$N$ has quadratic variation $[N,N]_t = \int_0^t \|\psi_s-\phi_s\|^2 \ell_{\tau_s} ds$ for all $t\in [0,T]$. Now, $N_t^2-[N,N]_t$ is a martingale having expectation $0$. Since $N_T=0$ $\tilde P$-a.s., it follows that $[N,N]_T = 0$ $\tilde P$-a.s., proving that for all almost every $t \in [0,T]$,  $\psi_t = \phi_t$ a.s. under $P$ and $\tilde P$. This proves the uniqueness.
\qed

\section{Laplace transforms of $\log(S_T)$ under $\inw P$ and $\check P$}\label{app:Laplace}

One can compute the Laplace transform or the characteristic function
of the conditional distribution of $\log\left(S_{T}/s\right)$ given
$S_{t}=s,{\tau}_{t}=i$ under $\inw P$ in the following way. Note that for any $0\le t\le T$,
%one can write
%\[
%\inw S_{T}=\inw S_{t}e^{\int_{t}^{T}\left\{ r_{\inw\tau_{u}}-\frac{1}{2}a_{\inw\tau_{u}}\right\} du+\int_{t}^{T}\sigma_{\inw\tau_{u-}}d\inw W_{u}},
%\]
%\[
%\check{S}_{T}=\check{S}_{t}e^{\int_{t}^{T}\left\{ r_{\tilde{\tau}_{u}}+\frac{1}{2}a_{\tilde{\tau}_{u}}\right\} du+\int_{t}^{T}\sigma_{\tilde{\tau}_{u-}}d\check{W}_{u}},
%\]
%for some Brownian motions $\inw W$ and $\check{W}$, since $\tilde{\tau}$
%and $\check{\tau}$ have the same distribution. In fact,
and for any $\theta\in\mathbb{C}$,
\begin{equation}
\inw{\Psi}_{T-t}(i,\theta)=E_{s,i}^{\inw P}\left\{ \left.e^{\theta\log\left(S_{T}/S_{t}\right)}\right|\calf_{t}\right\} =\frac{h_{i}\left\{ T-t,\theta-1,\frac{\theta(\theta-1)}{2}\right\} }{\delta_{i}(T-t)},\label{eq:cf}
\end{equation}
where, according to the Feynman-Kac formula in Theorem \ref{thm:FK},  for any $i\in\{1,\ldots,l\}$ and any $t>0 $,  $h_{i}(t,\theta_{1},\theta_{2})=E_{i}\left\{ e^{\int_{0}^{t}\{\theta_{1}r_{\tau_{u}}+\theta_{2}a_{\tau_{u}}-\ell_{\tau_{u}}\}du}\right\} $
solves %the following system of linear differential equation:
\begin{equation}
\dot h_i(t) = \frac{d}{dt}h_{i}(t)=(\theta_{1}r_{i}+\theta_{2}a_{i}-\ell_{i})h_{i}(t)+\sum_{j=1}^{l}\Lambda_{ij}h_{j}(t),\quad h_{i}(0)=1.\label{eq:laplace-sde}
\end{equation}
Then $h(t,\theta_{1},\theta_{2})=e^{t\{D(\theta_{1}r+\theta_{2}a-\ell)+\Lambda\}}\mathbf{1}$, $h_{i}(t,0,0)=\gamma_{i}(t)$, and $h_{i}(t,-1,0)=\beta_{i}(t)\gamma_{i}(t)$.
Similarly,
%\begin{equation}
$\check{\Psi}_{T-t}(i,\theta)=E_{i}^{\check P}\left\{ \left.e^{\theta\log\left(S_{T}/S_{t}\right)}\right|\calf_{t}\right\} =\frac{h_{i}\left\{ T-t,\theta,\frac{\theta(\theta+1)}{2}\right\} }{\gamma_{i}(T-t)}$.%\label{eq:cfcheck}
%\end{equation}

\end{document}